\documentclass[11pt,dvipsnames]{article} 
\usepackage{hyperref}
\usepackage[margin = 1in]{geometry}
\usepackage[utf8]{inputenc}
\usepackage{amsmath,amssymb,amsthm}

\usepackage{stmaryrd}

\usepackage{xcolor} \usepackage{enumitem}
\setlist[enumerate]{label=$(\alph*)$, leftmargin = *}
\usepackage{tikz} \usetikzlibrary{arrows} \tikzset{->, node distance =
  15mm, >=stealth', shorten > = 1pt, ArrowNode/.style =
  {font=\footnotesize}}
\usepackage{tikz-cd}
\theoremstyle{plain}
\newtheorem{lemma}{Lemma}[section]
\newtheorem{proposition}[lemma]{Proposition}
\newtheorem{corollary}[lemma]{Corollary}
\newtheorem{definition}[lemma]{Definition}
\newtheorem{theorem}[lemma]{Theorem}
\theoremstyle{definition}
\newtheorem{remark}[lemma]{Remark}
\newtheorem{example}[lemma]{Example}
\newcommand{\lra}{\leftrightarrows}
\newcommand{\ra}{\rightarrow}
\newcommand{\up}{{\uparrow}}
\newcommand{\mb}[1]{\mbox{#1}}
\newcommand{\ca}{\mathcal}
\newcommand{\se}{\subseteq}
\newcommand{\sm}{\setminus}
\newcommand{\we}{\wedge}
\newcommand{\ve}{\vee}

\newcommand{\bve}{\bigvee}

\newcommand{\bd}[1]{\mathbf{#1}}
\newcommand{\clf}{{\bf Clop}} 
\newcommand{\cl}{{\rm Clop}}
\newcommand{\Om}{\Omega}
\newcommand{\pt}{{\rm pt}}
\newcommand{\ptf}{{\bf pt}} 
\newcommand{\CUP}{{\bf CUp}}
\newcommand{\ko}{{\bf KO}}
\newcommand{\lperv}{{\bf L}_{\rm Pervin}}
\newcommand{\lfrith}{{\bf L}_{\rm Frith}}
\newcommand{\uperv}{{\bf U}_{\rm Pervin}}
\newcommand{\ufrith}{{\bf U}_{\rm Frith}}
\newcommand{\bbf}{{\bf B}} 
\newcommand{\bb}{{\rm B}}
\newcommand{\pp}{{\bf P}}
\newcommand{\psk}{{\bf Sk}_{\rm Pervin}} 
\newcommand{\fsk}{{\bf Sk}_{\rm Frith}} 
\newcommand{\psym}{{\bf Sym}_{\rm
    Perv}} 
\newcommand{\fsym}{{\bf Sym}_{\rm
    Frith}} 
\newcommand{\bpt}{{\bf pt}_b} 
\newcommand{\sem}[1]{[\,#1\,]_{se}}
\newcommand{\sei}[1]{[\,#1\,]_{op}}
\newcommand{\ba}[1]{\overline{#1}}
\newcommand{\cB}{{\mathcal B}} 
\newcommand{\cC}{{\mathcal C}}

\newcommand{\cK}{{\mathcal K}} 
\newcommand{\cL}{{\mathcal L}}
 
\newcommand{\cP}{{\mathcal P}}
\newcommand{\cS}{{\mathcal S}} 
\newcommand{\cT}{{\mathcal T}}
\newcommand{\cU}{{\mathcal U}}

\newcommand{\cX}{{\mathcal X}} 

\newcommand{\cN}{\mathcal N}
\newcommand{\idl}{{\rm Idl}}
\newcommand{\idlf}{{\bf Idl}} 

\newcommand{\two}{{\bf 2}}

\newcommand{\cffrm}{\bd{CFrith}}
\newcommand{\ffrmse}{\bd{Frith}_{se}}
\newcommand{\pervse}{\bd{Pervin}_{se}}
\newcommand{\cperv}{\bd{CPervin}}
\newcommand{\spec}{{\bf Spec}}

\newcommand{\Frm}{{\bf Frm}}
\newcommand{\bifrm}{{\bf BiFrm}}
\newcommand{\kzbitop}{{\bf BiTop}_{\rm KZ}}
\newcommand{\kzbifrm}{{\bf BiFrm}_{\rm KZ}}
\newcommand{\bitop}{{\bf BiTop}}
\newcommand{\ffrm}{{\bf Frith}}
\newcommand{\sffrm}{{\bf Frith}_{\rm sym}}

\newcommand{\dlat}{{\bf DLat}}
\newcommand{\perv}{{\bf Pervin}}
\newcommand{\sperv}{{\bf Pervin}_{\rm sym}}
\newcommand{\Top}{{\bf Top}}
\newcommand{\pri}{\bf Priest}
\newcommand{\pf}{{\rm pf}}
\newcommand{\pff}{{\bf pf}} 

\begin{document}

\title{Pervin spaces and Frith frames:\\ bitopological aspects and
  completion} \author{C\'{e}lia Borlido\thanks{Centre for Mathematics of the University of Coimbra (CMUC), 3001-501 Coimbra, Portugal, cborlido@mat.uc.pt} \and Anna Laura Suarez\thanks{Laboratoire J. A. Dieudonn\'{e} UMR CNRS 7351, Universit\'{e} Nice Sophia Antipolis, 06108 Nice Cedex 02, France,
annalaurasuarez993@gmail.com (corresponding author)
}} \date{}

\maketitle

\abstract{A Pervin space is a set equipped with a bounded sublattice
  of its powerset, while its pointfree version, called Frith frame,
  consists of a frame equipped with a generating bounded
  sublattice. It is known that the dual adjunction between topological
  spaces and frames extends to a dual adjunction between Pervin spaces
  and Frith frames, and that the latter may be seen as representatives
  of certain quasi-uniform structures. As such, they have an
  underlying bitopological structure and inherit a natural notion of
  completion.  In this paper we start by exploring the bitopological
  nature of Pervin spaces and of Frith frames, proving some
  categorical equivalences involving zero-dimensional structures. We
  then provide a conceptual proof of a duality between the categories
  of $T_0$ complete Pervin spaces and of complete Frith frames. This
  enables us to interpret several Stone-type dualities as a
  restriction of the dual adjunction between Pervin spaces and Frith
  frames along full subcategory embeddings. Finally, we provide
  analogues of Banaschewski and Pultr's characterizations of sober and
  $T_D$ topological spaces in the setting of Pervin spaces and of
  Frith frames, highlighting the parallelism between the two notions.}

\section{Introduction}
The category of Pervin spaces is introduced
in~\cite{GehrkeGrigorieffPin2010, pin17} as an isomorph of the
category of transitive and totally bounded quasi-uniform spaces. Its
pointfree analogue, whose objects are named Frith frames, was later
defined in~\cite{borlido21}. In this setting, we have a full embedding
of the category of Frith frames in that of transitive and totally
bounded quasi-uniform frames, which is a coreflection but not an
equivalence. It is also shown in~\cite{borlido21} that the classical
dual adjunction between topological spaces and frames naturally
extends to a dual adjunction between Pervin spaces and Frith
frames. In fact, this is what justifies calling Frith frames the
pointfree version of Pervin spaces.

Since both Pervin spaces and Frith frames may be seen as quasi-uniform
structures, they come equipped with an underlying bitopological
structure as well~\cite[Chapter~3]{Frith86}. The study of such
bitopological structure is the main content of
Section~\ref{sec:bitop}. In Section~\ref{sec:se}, we start by
assigning a bitopological space to each Pervin space, and show that
this is a functorial assignment with a left adjoint. When studying the
categorical equivalence induced by such adjunction, \emph{strong
  exactness} (for Pervin spaces) and \emph{zero-dimensionality} (for
bitopological spaces) appear as crucial concepts. More precisely, we
show that the categories of the so-called \emph{strongly exact Pervin
  spaces} and of \emph{zero-dimensional bitopological spaces} are
equivalent. We then consider the pointfree version of these results
and show that \emph{strongly exact Frith frames} are a full
coreflective subcategory of the category of \emph{zero-dimensional
  biframes}, leaving as an open problem to describe the underlying
categorical equivalence. Noting that topological spaces and frames may
be seen as bitopological spaces and biframes, respectively, in
Section~\ref{sec:zero-dim}, we specialize the results of the previous
section in the monotopological setting. In particular, we show that
the categories of \emph{zero-dimensional topological spaces} and of
\emph{strongly exact symmetric Pervin spaces} are equivalent, and so
are those of \emph{zero-dimensional frames} and of \emph{strongly
  exact symmetric Frith frames}.

As representatives of quasi-uniform structures, Pervin spaces and
Frith frames also inherit natural notions of \emph{completeness}. It
is observed in~\cite{pin17} that $T_0$ and complete Pervin spaces can
be identified with spectral spaces, while in~\cite{borlido21} it is shown
that complete Frith frames can be identified with bounded distributive
lattices. In particular, thanks to Stone duality for bounded
distributive lattices, it follows that the categories of $T_0$ and
complete Pervin spaces and of complete Frith frames are dual to each
other. In Section~\ref{sec:completion}, we provide a direct and
conceptual proof of this duality, which is based on a characterization
of \emph{complete Pervin spaces} and of \emph{complete Frith frames},
and does not invoke Stone duality.  On the other hand, since the
categories of $T_0$ complete Pervin spaces and of complete Frith
frames are full subcategories of the categories of Pervin spaces and
of Frith frames, respectively, we may then see Stone duality as a
restriction of the Pervin-Frith dual adjunction along \emph{full}
subcategory embeddings (unlike what happens when seeing it as a
restriction of the dual adjunction between topological spaces and
frames). In Section~\ref{sec:duality}, we exhibit several Stone-type
dualities as suitable restrictions of the dual adjunction along
\emph{full} subcategory embeddings. Section~\ref{sec:stone} is devoted
to the already mentioned Stone duality, Section~\ref{sec:priest} to
Priestley duality, and Section~\ref{subsec:bitopdua} to bitopological
duality. In Section~\ref{sec:overview}, we provide the global picture
of the results thus obtained. It is our concern in
Sections~\ref{sec:completion} and~\ref{sec:duality} to point out where
the assumption of the \emph{Prime Ideal Theorem} is needed.

Finally, in Section~\ref{sec:last}, starting from Banaschewski and
Pultr's characterizations of sober and $T_D$ topological
spaces~\cite[Proposition 4.3]{banaschewski10}, which highlights the
parallelism between the two notions, we state and prove analogous
results for Pervin spaces, where \emph{sober} is replaced by
\emph{complete} and \emph{$T_D$} by its equivalent for Pervin spaces (the latter
notion having been introduced in~\cite[Section~4.5]{borlido21}). When
looking for a pointfree version of such results, we are naturally led
to consider the notion of \emph{locale-based Frith frame}, which will
replace \emph{$T_D$ Pervin space} in our statement.

We readily warn the reader that, although we implicitly have in mind
the quasi-uniform interpretation of Pervin spaces and of Frith frames
(namely, when considering their bitopological nature and the property
of being complete), we will avoid mentioning quasi-uniformities
throughout the paper. This reduces the amount of background required
from the reader, leaving the details of existing connections for those
already familiar with quasi-uniformities. For a detailed study of
Pervin spaces, Frith frames, and corresponding quasi-uniform
structures, we refer to~\cite{borlido21, GehrkeGrigorieffPin2010,
  pin17}.
\section{Preliminaries}
The material in this section is presented mostly to set up the
notation.  We assume the reader to be familiar with frame and locale
theory.

\subsection{Basic notation}
The content of this section may be found in~\cite{johnstone82,
  picadopultr2011frames}.

A \emph{frame} is a complete lattice~$L$ such that for every $a \in L$
and $\{b_i\}_{i \in I} \subseteq L$ the following distributivity law
holds:
\[a \wedge \bigvee_{i \in I} b_i = \bigvee_{i \in I}(a \wedge b_i).\]
A frame $L$ is always a complete Heyting algebra, with the
\emph{Heyting implication} given by
\[
  a\ra b=\bve \{x\in L\mid x\we a\leq b\},
\]
for every $a, b \in L$. The element $a \to 0$, called the
\emph{pseudocomplement of $a$}, will be denoted by~$a^*$.  When $a
\vee a^* = 1$, we say that $a$ is \emph{complemented} and $a^*$ is the
\emph{complement} of~$a$. A \emph{frame homomorphism} is a map $h: L
\to M$ that preserves finite meets and arbitrary joins. We will denote
by $\Frm$ the category of frames and frame homomorphisms.  A frame
homomorphism~$h: L \to M$ is \emph{dense} provided $h(a) = 0$ implies
$a = 0$. In general, a frame homomorphism need not preserve the
Heyting implication. We have however the following:
\begin{equation}
  \label{eq:1}
  h(a \to b) \leq h(a) \to h(b),
\end{equation}
for every $a, b \in L$.  Moreover, since every frame homomorphism $h$
preserves arbitrary joins, it has a right adjoint $h_*$, and the
equality
\begin{equation}
  \label{eq:2}
  a \to h_*(x) = h_*(h(a) \to x)
\end{equation}
holds, for every $a \in L$ and $x \in M$. This is usually called the
\emph{Frobenius identity}.

We say that an element $a \in L$ is \emph{compact} if whenever $a \leq
\bigvee_{i \in I}a_i$ there exists a finite subset $I' \subseteq I$
such that $a \leq \bigvee_{i \in I'} a_i$. A frame~$L$ is
\emph{compact} if its top element is compact.  If the set of compact
elements of~$L$ is closed under finite meets and join-generates~$L$,
then we say that $L$ is \emph{coherent}. A frame $L$ is
\emph{zero-dimensional} if it is join-generated by its sublattice of
complemented elements.
  
The opposite category of $\Frm$, is usually denoted by $\bf Loc$. Its
objects are called \emph{locales}, and morphisms $h:L\to M$ are the
right adjoints of the corresponding frame maps. Locales will only be
mentioned in Section~\ref{sec:last}. For our purposes, the notion of
\emph{sublocale} will be enough. A \emph{sublocale} of $L$ is a subset
$K \subseteq L$ that is closed under arbitrary joins and contains
every element of the form $a \to x$, with $a \in L$ and $x \in K$. A
sublocale $K$ is itself a frame, but not a \emph{subframe of $L$}, as
joins may be computed differently. Every sublocale is uniquely
determined by a frame quotient $q: L \twoheadrightarrow K$ whose right
adjoint is the localic embedding $K\hookrightarrow L$. In particular,
$q (x) = x$, whenever $x \in K$.

Since sublocales correspond to frame quotients, these may also be
defined via \emph{frame congruences}, which will be widely used
throughout the paper. The set $\cC L$ of all frame congruences on~$L$
is itself a frame when ordered by inclusion. For every element $a \in
L$, we may define two congruences:
\[\nabla_a := \{(x,y)\in L\times L\mid a\ve x=a\ve y\} \qquad
  \text{and}\qquad\Delta_a := \{(x,y)\in L\times L\mid a\we x=a\we
  y\}.\]
Congruences of the form $\nabla_a$ are called \emph{closed}, while
those of the form $\Delta_a$ are \emph{open}.  Open and closed
congruences suffice to generate~$\cC L$, as a frame. The functions
$\nabla: a \mapsto \nabla_a$ and $\Delta: a \mapsto \Delta_a$ from $L$
to $\cC L$ are, respectively, a frame embedding and an injection that
turns finite meets into finite joins and arbitrary joins to
arbitrary meets. Given a subset $S \subseteq L$, we denote by $\nabla
S$ and by $\Delta S$ the subframes of $\cC L$ generated by $\{\nabla_s
\mid s \in S\}$ and by $\{\Delta_s \mid s \in S\}$, respectively. The
subframe of $\cC L$ generated by $\nabla L \cup \Delta S$ will be
denoted by $\cC_SL$.  The following generalizes the well-known
universal property of the congruence frame.

\begin{proposition}[{\cite[Theorem~16.2]{Wilson1994TheAT}}]\label{p:10}
  For every frame $L$ and subset $S\se L$, the frame $\cC_S L$ has the
  following universal property: whenever $h:L\ra M$ is a frame map
  such that $h(s)$ is complemented for all $s\in S$, there is a unique
  frame homomorphism $\widetilde{h}:\cC_S L\ra M$ making the following
  diagram commute.
  \begin{center}
    \begin{tikzpicture}
      [->, node distance = 20mm] \node (A) {$L$}; \node[right of = A]
      (B) {$\cC_S L$}; \node[below of = B, yshift = 2mm] (C) {$M$};
      \draw[right hook->] (A) to node[above, ArrowNode] {$\nabla$} (B);
      \draw[dashed] (B) to node[right, ArrowNode] {$\widetilde h$}
      (C); \draw (A) to node[below, ArrowNode] {$h$}(C);
    \end{tikzpicture}
  \end{center}
\end{proposition}

Finally, we have an idempotent adjunction $\bf \Omega: \Top \lra {\bf
  Loc}: \ptf$ between the category $\Top$ of topological spaces and
continuous functions and the category of locales. Since, as already
mentioned, we will mostly work with frames, we will treat the category
$\bf Loc$ as that opposite to $\Frm$. Given a
topological space $(X, \tau)$ (or simply $X$ if no confusion arises),
${\bf \Omega}(X)$ is the frame $\Omega(X)$ consisting of the open
subsets of $X$ (ordered by $\se$) and, for a continuous function
$f:X\to Y$, ${\bf \Omega}(f) = f^{-1}$ is the preimage frame
homomorphism. In the other direction, given a frame $L$, $\ptf(L)$ is
the set $\pt(L)$ of \emph{points} of~$L$ (here seen as frame
homomorphisms $p: L \to \two$) equipped with the topology $\widehat L
:= \{\widehat a\mid a \in L\}$, where $\widehat a := \{p \in \pt(L)
\mid p(a) = 1\}$, and given a frame homomorphism $h: L \to M$,
$\ptf(h)$ maps $p \in \pt(M)$ to $p \circ h \in \pt(L)$. The fixpoints
of ${\bf\Omega} \dashv \ptf$ are the so-called \emph{sober spaces} and
\emph{spatial frames}, respectively.

\subsection{Bitopological spaces and biframes}

We refer to~\cite{salbany74} and to~\cite{banaschewski83, schauerte92} for
further reading on bitopological spaces and on biframes, respectively.

A \emph{bitopological space}, or \emph{bispace}, is a triple $\cX =
(X,\tau_+,\tau_-)$, where $X$ is a set and $\tau_+$ and $\tau_-$ two
topologies on that set. Morphisms between bitopological spaces are
functions between their underlying sets that are continuous with
respect to both topologies. We denote by $\bitop$ the category thus
obtained. The topology $\tau_+ \vee \tau_-$ on $X$ generated by
$\tau_+\cup \tau_-$ is called the \emph{patch} topology. We call
$\tau_+$ the \emph{positive} topology, and $\tau_-$ the
\emph{negative} one. Accordingly, elements of $\tau_+$ are called
\emph{positive opens}, and a positive open whose complement is a
negative one is called a \emph{positive clopen}. The collection of
positive clopen subsets of $X$ will be denoted by
$\cl_+(\cX)$. Negative (cl)opens and $\cl_-(\cX)$ are defined
similarly, in the obvious way.

We say that a bispace is \emph{$T_0$} (respectively, \emph{compact})
if its patch topology is $T_0$ (respectively, compact). We say that a
bispace $(X,\tau_+,\tau_-)$ is \emph{zero-dimensional} if every
element in $\tau_+$ is a union of positive clopens, and every element
in $\tau_-$ is a union of negative clopens.\footnote{These notions are
  not consistent over all the literature. For instance,
  in~\cite{bezhanishvili10} our notions of $T_0$ and compact are named
  \emph{join $T_0$} and \emph{join compact}, respectively, while
  compact is named \emph{pairwise compact}
  in~\cite{salbany74}. Moreover, in~\cite{Reilly73, bezhanishvili10}
  our notion of zero-dimensional is named \emph{pairwise
    zero-dimensional}.}

A \emph{biframe} is a triple $\cL=(L,\, L_+,\,L_-)$ such that all
three components are frames, together with subframe inclusions $L_+\se
L$ and $L_-\se L$, and such that every element of $L$ is a join of
finite meets of elements of $L_+\cup L_-$. The frame $L$ is called
the \emph{main component} of the biframe, while $L_+$ and $L_-$ are,
respectively, the \emph{positive} and \emph{negative}
components. Accordingly, elements of $L_+$ are \emph{positive} and
those of $L_-$ are \emph{negative}. We say that an element $a \in L_+$
is \emph{positive bicomplemented} when it is complemented in~$L$ with
complement in $L_-$. \emph{Negative bicomplemented} elements are
defined similarly. We denote by $\bb_+(\cL)$ and by $\bb_-(\cL)$ the
lattices of positive and negative bicomplemented elements of~$\cL$,
respectively. A morphism $h:(L,\, L_+,\,L_-)\ra (M,\,M_+,\,M_-)$
between biframes is a frame homomorphism $h:L\ra M$ such that
$h[L_+]\se M_+$ and $h[L_-]\se M_-$. The category of biframes and
biframe homomorphisms will be denoted by $\bifrm$. We say that~$h$ is
\emph{dense} if its underlying frame homomorphism is dense.

We say that a biframe $\cL = (L, L_+, L_-)$ is \emph{compact} if its
main component~$L$ is compact, and that it is \emph{zero-dimensional} if
both $L_+$ and $L_-$ are join-generated by their bicomplemented
elements.

Finally, we also have an adjunction ${\bf \Omega}_b: \bitop \lra
\bifrm^{op}: \ptf_b$ between bitopological spaces and biframes, which
extends the classical adjunction between topological spaces and frames
(here, we identify a topological space $(X, \tau)$ with the bispace
$(X, \tau, \tau)$, and a frame $L$ with the biframe $(L, L,
L)$). Given a bitopological space $\cX = (X, \tau_+, \tau_-)$, ${\bf
  \Omega}_b(\cX)$ is the biframe $(\tau_+ \vee \tau_-, \tau_+,
\tau_-)$, while for a biframe $\cL = (L, L_+, L_-)$, $\ptf_b(\cL) =
(\pt(L), \widehat {L_+}, \widehat{L_-})$, where for $P \in \{L_+,
L_-\}$, we denote $\widehat P := \{\widehat a \mid a \in P\}$. On
morphisms, ${\bf \Omega}_b$ and $\ptf_b$ are defined as expected.
\subsection{The Prime Ideal Theorem}

One of the main features of \emph{pointfree topology} is that it often
avoids the use of the Axiom of Choice, thereby leading to constructive
results. Sometimes, a strictly weaker version known as the
\emph{Prime Ideal Theorem} is however needed. It may be stated as
follows:
\begin{quote}
  Let $D$ be a bounded distributive lattice. Then, every proper ideal
  of~$D$ may be extended to a prime ideal. 
\end{quote}
In Sections~\ref{sec:completion} and~\ref{sec:duality}, some results
are valid only under the assumption of the Prime Ideal Theorem. The
following equivalent statements will be useful in there.
\begin{theorem}\label{pit} The following
  statements are equivalent:
\begin{enumerate}
    \item The Prime Ideal Theorem holds.
    \item For every set $X$, every proper filter of $\cP(X)$ can be
      extended to an ultrafilter.
    \item Every frame of the form $\idl(D)$, where $D$ is a bounded
      distributive lattice, is spatial.
\end{enumerate}
\end{theorem}
Relevant references for the equivalence between these statements
are~\cite{banaschewski1981, MR1058437, jech73, johnstone82}.

\subsection{Pervin spaces and Frith frames}\label{sec:perv-frith}

A \emph{Pervin space} is a pair $(X,\cS)$ where $X$ is a set and $\cS$
a sublattice of its powerset. A morphism $f:(X, \cS) \to (Y, \cT)$ of
Pervin spaces is a set function $f:X \to Y$ such that $f^{-1}(\cT)
\subseteq \cS$.  Since a topology on a set~$X$ is, in particular, a
bounded sublattice of its powerset, the category $\Top$ of topological
spaces fully embeds in the category $\perv$ of Pervin spaces.
\begin{proposition}
  [{\cite[Section~3]{borlido21}}] Let $f: (X, \cS) \to (Y, \cT)$ be a
  map of Pervin spaces. Then,
  \begin{enumerate}
  \item $f$ is an epimorphism if and only if its underlying set map is
    surjective;
  \item $f$ is an extremal monomorphism if and only if its underlying
    set map is injective and every element of $\cS$ is of the form
    $f^{-1}(T)$ for some $T\in \cT$.
  \end{enumerate}
  In particular, $f$ is an isomorphism if and only if its underlying
  set map is a bijection and $f[\cS] = \cT$.
\end{proposition}

Recall that a \emph{join-dense} subset of a complete lattice $L$ is a subset whose closure under (arbitrary) joins is $L$ itself. A \emph{Frith frame} is a pair $(L, S)$, where $L$ is a frame and $S$
is a join-dense bounded sublattice of~$L$. A morphism of Frith frames
$h: (L, S) \to (M, T)$ is a frame homomorphism between the underlying
frames that satisfies $h[S] \subseteq T$. The category of Frith frames
and their morphisms is denoted $\ffrm$, and we have a full embedding $\Frm
\hookrightarrow \ffrm$ obtained by identifying a frame $L$ with the
Frith frame $(L, L)$.
\begin{proposition}
  [{\cite[Section~4.4]{borlido21}}] \label{p:4} Let $h: (L, S) \to (M,
  T)$ be a homomorphism of Frith frames. Then,
  \begin{enumerate}
  \item $h$ is a monomorphism if and only if so is its underlying frame
    homomorphism;
  \item $h$ is an extremal epimorphism if and only if it satisfies
    $h[S] = T$.
  \end{enumerate}
  In particular, $h$ is an isomorphism if and only if it is injective
  and satisfies $h[S] = T$.
\end{proposition}

The classical dual adjunction between topological spaces and frames
may then be extended to a dual adjunction ${\bf \Omega}: \perv \lra
\ffrm: \ptf$ between Pervin spaces and Frith frames as follows.  For a
Pervin space $(X,\cS)$, ${\bf \Omega}(X,\cS)$ is the Pervin space
$(\Om_{\cS}(X),\cS)$, where $\Om_{\cS}(X)$ denotes the topology on~$X$
generated by $\cS$. For a morphism $f:(X,\cS)\ra (Y,\cT)$, ${\bf
  \Omega}(f)$ is the preimage map $f^{-1}: (\Omega_\cT(Y), \cT) \to
(\Omega_\cS(X), \cS)$.  In the other direction, for a Frith frame
$(L,S)$, $\ptf(L,S)$ is the Pervin space $(\pt(L), \widehat S)$, where
$\widehat S := \{\widehat s\mid s \in S\}$. Finally, given a morphism
of Frith frames $h: (L, S) \to (M, T)$, $\ptf(h)$ maps $p \in \pt(M)$
to $p \circ h \in \pt(L)$.

\begin{theorem}[{\cite[Proposition 4.3]{borlido21}}]\label{dualadj}
  There is an adjunction ${\bf \Omega}:\perv\lra \ffrm^{op}:\ptf$ with
  ${\bf \Omega}\dashv \ptf$, whose fixpoints are, respectively, the
  Pervin spaces $(X, \cS)$ such that $(X,\, \Omega_\cS(X))$ is
  sober and the Frith frames $(L, S)$ such that $L$ is spatial. These will be called, respectively, \emph{sober
    Pervin spaces} and \emph{spatial Frith frames}.
\end{theorem}
\subsection{Symmetrization}
\emph{Symmetrization} is for Pervin spaces and Frith frames, what
\emph{uniformization} is for quasi-uniform spaces and quasi-uniform
frames, respectively. This has been considered in~\cite{pin17} for
Pervin spaces and in~\cite{borlido21} for Frith frames. In particular,
it has been shown that the categories of \emph{symmetric} Pervin
spaces and of \emph{symmetric} Frith frames are equivalent to the
categories of transitive and totally bounded uniform spaces and
frames, respectively.

Recall that a Pervin space $(X,\cB)$ is \emph{symmetric} if $\cB$ is a
Boolean algebra, and the full subcategory of $\perv$ determined by the
symmetric Pervin spaces will be denoted by $\sperv$. A Frith frame
$(L, B)$ is \emph{symmetric} if $B$ is a Boolean algebra, and the full
subcategory of $\ffrm$ determined by the symmetric Frith frames will
be denoted by $\sffrm$.

We define a functor $\psym:\perv\ra \sperv$ as follows. For an object
$(X,\cS)$, we let $\psym(X,\cS)$ be the Pervin space $(X,\ba\cS)$,
where $\ba{\ca{S}}$ is the Boolean subalgebra of the powerset
$\ca{P}(X)$ generated by the elements of~$\ca{S}$.  On morphisms, we
simply map a function to itself.

\begin{proposition}[{\cite[Proposition~3.4]{borlido21}}]\label{p:1}
  The functor $\psym$ is right adjoint to the embedding $\sperv
  \hookrightarrow \perv$.
\end{proposition}

The pointfree version of~$\psym$ is defined as follows.  For a Frith
frame $(L,S)$ we set $\fsym(L,S) := (\cC_S L,\ba S)$, where $\ba S$
denotes the sublattice of $\cC_S L$ generated by the elements of the
form $\nabla_s$ together with their complements. For a morphism of
Frith frames $h: (L,S)\ra (M,T)$ we set $\fsym(h) := \overline h$,
where $\overline h$ is the unique extension of~$h$ to a frame
homomorphism $\overline{h}: \cC_SL \to \cC_TM$ (recall
Proposition~\ref{p:10}).

\begin{proposition}[{\cite[Proposition~6.5]{borlido21}}]\label{p:25}
  The functor $\fsym$ is left adjoint to the embedding $\sffrm
  \hookrightarrow \ffrm$.
\end{proposition}

\subsection{Forgetful functors}\label{sec:forget}
Both Pervin spaces and Frith frames encode two kinds of structures -
topological and lattice-theoretical. It is then natural to consider
several forgetful functors on the categories $\perv$ and $\ffrm$.

Every Pervin space $(X,\cS)$ defines a topology $\Omega_\cS(X)$
on~$X$. This assignment can be extended to a functor $\uperv:\perv\ra
\Top$ which leaves functions unaltered.\footnote{Under the
  identification of Pervin spaces with transitive and totally bounded
  quasi-uniform spaces, this functor forgets the quasi-uniform
  structure.} The pointfree version of this functor is the functor
$\ufrith:\ffrm\ra \Frm$, which acts as $(L,S)\mapsto L$ on objects,
and which leaves frame maps unaltered. The relationship between
$\uperv$ and $\ufrith$ is depicted in the following commutative
diagrams:
\begin{center}
  \begin{tikzpicture}
    \node (A) {$\perv$}; \node[right of = A, xshift = 20mm] (B)
    {$\Top$}; \node[below of = A] (C) {$\ffrm^{op}$}; \node[below of =
    B] (D) {$\Frm^{op}$};
    \draw (A) to node[ArrowNode,above] {$\uperv$} (B); \draw (B) to
    node[ArrowNode,right] {$\bf \Omega$} (D); \draw (A) to
    node[ArrowNode,left] {$\bf \Omega$} (C); \draw (C) to
    node[ArrowNode, below] {$\ufrith$} (D);
    \node[right of = A, xshift = 50mm] (A') {$\perv$}; \node[right of
    = A', xshift = 20mm] (B') {$\Top$}; \node[below of = A'] (C')
    {$\ffrm^{op}$}; \node[below of = B'] (D') {$\Frm^{op}$};
    \draw (A') to node[ArrowNode,above] {$\uperv$} (B'); \draw (D') to
    node[ArrowNode,right] {$\ptf$} (B'); \draw (C') to
    node[ArrowNode,left] {$\ptf$} (A'); \draw (C') to node[ArrowNode,
    below] {$\ufrith$} (D');
  \end{tikzpicture}
\end{center}

\begin{proposition}[{{\cite{borlido21}}}]
  The functor $\uperv$ is right adjoint to the embedding $\Top
  \hookrightarrow \perv$, and the functor $\ufrith$ is left adjoint to
  the embedding $\Frm \hookrightarrow \ffrm$.
\end{proposition}

Given a topological property, we will often say that a Pervin space
has that property provided so does its underlying topological
space. For instance, a Pervin space $(X, \cS)$ is \emph{$T_0$} if $(X,
\Omega_\cS(X))$ is $T_0$. The same applies to the notion of
\emph{dense} morphism. We say that a morphism $f: (X, \cS) \to (Y,
\cT)$ is \emph{dense} if the image of $\uperv(f): (X, \Omega_\cS(X))
\to (Y, \Omega_\cT(Y))$ is dense in $(Y, \Omega_{\cT}(Y))$, that is,
if $f[X]$ intersects every nonempty open subset of $(Y,
\Omega_\cT(Y))$. The following characterization of density for Pervin
morphisms is easy to prove. We will use it without further mention.
\begin{lemma}\label{l:9}
  For a map $f:(X,\cS)\ra (Y,\cT)$ of Pervin spaces, the following are
  equivalent.
  \begin{enumerate}
  \item \label{denseb} The map $f$ is dense;
  \item \label{densea} $f[X]$ intersects every nonempty element of
    $\cT$;
  \item \label{densed} We have $f^{-1}(T)=\emptyset$ implies
    $T=\emptyset$ for all $T\in \cT$.
  \end{enumerate}
\end{lemma}

We may also forget the topological structures, thereby obtaining
functors $\lperv: \perv \to \dlat^{op}$ and $\lfrith: \ffrm \to \dlat$
defined in the expected way.

In~\cite[Proposition~4.6]{borlido21}, we have proved that $\lfrith$ is
right adjoint to the functor $\idlf:\dlat\ra \ffrm$ acting on objects
as $S\mapsto (\idl(S),S)$ and assigning to each morphism $f:S\ra T$
its unique extension to a frame morphism $\idl(S)\ra \idl(T)$. The
component of the counit of this adjunction at a Frith frame $(L, S)$
is the morphism
\begin{equation}
  c_{(L, S)}: (\idl(S), S) \to (L, S), \qquad J \mapsto \bigvee
  J.\label{eq:5}
\end{equation}

We now discuss the point-set version of this result. To define the
right adjoint of~$\lperv$, we introduce some notation. Given a lattice
$D$, we denote by $\pf(D)$ the set of all prime filters of~$D$, and we
consider the map
\begin{equation}
  \Phi_D: D \to \cP(\pf(D)),\qquad a \mapsto \widetilde a:=  \{F \in \pf(D)
  \mid a \in F\}.\label{eq:4}
\end{equation}
We denote as $\widetilde D := \{\widetilde a \mid a \in D\}$.  It is
well-known that $\Phi_D$ is a lattice homomorphism and it is an
embedding if and only if $D$ is isomorphic to a sublattice of the
powerset~$\cP(X)$ for some set~$X$ (see
e.g. \cite[Chapter~10]{MR1058437}). If we further assume that the Prime
Ideal Theorem holds, then $\Phi_D$ is \emph{always} an embedding.

We define the functor $\pff:\dlat^{op}\ra \perv$ as the one mapping
each lattice $D$ to the Pervin space $(\pf(D),\widetilde D)$, and each
lattice homomorphism $f:D_1\ra D_2$ to the preimage map
$\pff(f):=f^{-1}:(\pf(D_2), \widetilde{D_2})\ra (\pf(D_1),
\widetilde{D_1})$. This is well-defined as preimages of prime filters
are prime filters, and $\pff(f)^{-1}(\widetilde a) = \widetilde {f(a)}$
for every $a \in D_1$.

\begin{lemma}\label{lvdashpf}
  We have an idempotent adjunction $\lperv:\perv\lra \dlat^{op}:\pff$
  with $\lperv\dashv \pff$.
\end{lemma}
\begin{proof}
  With a routine computation one can show that, given a lattice $D$,
  the co-restriction of $\Phi_D$ to a map $D \to \widetilde D$ is a
  universal morphism from $\lperv$ to $D$ and thus, the counit of the
  adjunction at~$D$. To conclude that the adjunction is idempotent, we
  only have to note that $\Phi_D$ is an isomorphism whenever $D$ is a
  sublattice of some powerset and that is the case for every lattice
  component of a Pervin space $(X, \cS)$.
\end{proof}
We further remark that the component of the unit of the adjunction
$\lperv \dashv \pff$ at a Pervin space $(X, \cS)$ is the neighborhood
morphism of Pervin spaces defined by
\begin{equation}
  \cN_{(X, \cS)}: (X, \cS) \to (\pf(\cS), \widetilde \cS), \qquad x
  \mapsto \{S \in \cS \mid x \in S\}.\label{eq:3}
\end{equation}

We finish this section by providing an alternative representation of
the Pervin space $\pff(D)$, based on the well-known correspondence
between prime filters of $D$ and points of $\idl(D)$.

\begin{lemma}\label{l:8}
  There is a one-to-one correspondence between prime filters of $D$
  and points of $\idl(D)$ and, under this identification, $\widetilde
  a$ corresponds to $\widehat a$, for every $a\in D$. In particular,
  $\pff(D)$ is isomorphic to $\ptf(\idl(D), D)$.
\end{lemma}

\subsection{Open intersections and strongly exact meets}

For a Pervin space $(X,\cS)$ we say that an intersection of elements
in $\cS$ is \emph{open} if it is so in the topological space~$(X,
\Omega_\cS(X))$. We will denote by $\sei \cS$ the collection of all
elements of $\Omega_\cS(X)$ that are open intersections of elements of
$\cS$. It is easy to see that this collection is closed under open
intersections and so, $\sei \cS$ may be seen as the closure of $\cS$
under open intersections.

Let now $L$ be a frame and $P \subseteq L$. The meet $\bigwedge P$ is
\emph{strongly exact} if the corresponding intersection of open
sublocales (=open pointfree subspaces) is open, or equivalently, if
the congruence $\bigvee_{s \in P} \Delta_s$ is open
(cf. \cite[Section~4.5]{Ball14}).  Note that, if $\Delta_a =
\bigvee_{s \in P} \Delta_s$ for some $P \subseteq L$, then the meet
$\bigwedge P$ is, by definition, strongly exact, and we necessarily
have $a = \bigwedge P$.  Given a Frith frame $(L, S)$, we denote by
$\sem S$ the set of elements of $L$ that may be written as a strongly
exact meet of elements of~$S$. Again, $\sem S$ can be thought of as
the closure of $S$ under strongly exact meets.

Finally, we say that a Pervin space $(X, \cS)$ (respectively, Frith
frame $(L,S)$) is \emph{strongly exact} if the lattice $\cS$
(respectively, $S$) is closed under open intersections of~$\cS$
(respectively, strongly exact meets of $S$). We denote by $\pervse$
and by $\ffrmse$ the full subcategories of $\perv$ and of $\ffrm$
determined by the strongly exact objects.

The next result implies that the contravariant functor ${\bf \Omega}:
\perv \to \ffrm$ restricts and co-restricts to a functor $\pervse \to
\ffrmse$.
\begin{proposition}[{\cite[Proposition 5.3]{Ball14}}]\label{p:9}
  Let $(X, \tau)$ be a topological space and $\cU \subseteq \Omega(X)$
  be a family of open subsets. If $\bigwedge \cU$ is a strongly exact
  meet in the frame $\Omega(X)$, then $\bigcap \cU$ is an open subset
  of~$X$.
\end{proposition}

We do not know whether the functor $\ptf: \ffrm \to \perv$ restricts
and co-restricts to a functor $\ffrmse \to \pervse$.

\section{The bitopological point of view}\label{sec:bitop}
\subsection{Strong exactness and zero-dimensionality}\label{sec:se}

A Pervin space $(X,\ca{S})$ defines the bitopological space
$(X,\,\Om_{\cS}(X),\,\Om_{\cS^c}(X))$, where $\cS^c$ denotes the
lattice $\{S^c \mid S \in \cS\}$ formed by the complements in~$X$ of
the elements of~$\cS$.\footnote{For the interested reader, this is the
  underlying bitopological space of the quasi-uniform space
  represented by $(X, \cS)$.} The \emph{Skula functor} $\psk:\perv\ra
\bitop{}$ is then defined by assigning
$(X,\,\Om_{\cS}(X),\,\Om_{\cS^c}(X))$ to the Pervin space
$(X,\ca{S})$, and mapping each function to itself.

In the other direction, we may define a functor $\clf_+: \bitop \to \perv$ by
assigning to each bitopological space $\cX = (X, \tau_+, \tau_-)$ the
Pervin space $(X, \,\cl_+(\cX))$ and keeping morphisms
unchanged.

It is easily seen that these are well-defined functors. Let us prove
that $\clf_+$ is left adjoint to $\psk$.

\begin{lemma}\label{biadj}
  The functor $\clf_+$ is left adjoint to $\psk: \perv \to \bitop$.
\end{lemma}
\begin{proof}
  Let $\cX = (X, \tau_+, \tau_-)$ be a bitopological space. We first
  observe that
  \[\psk \circ \clf_+(\cX) = (X,
    \,\Omega_{\cl_+(\cX)}(X),\, \Omega_{\cl_-(\cX)}(X)).\]
  Since the inclusions $\cl_+(\cX) \subseteq \tau_+$ and $\cl_-(\cX)
  \subseteq \tau_-$ hold, the identity function on~$X$ induces a
  morphism of bitopological spaces $\eta_{\cX}: \cX \to \psk \circ
  \clf_+(\cX)$. Let us show that $(\clf_+(\cX),\, \eta_{\cX})$ is a
  universal morphism from $\cX$ to $\psk$.  Let $(Y, \cT)$ be a Pervin
  space and $f: \cX \to \psk(Y, \cT)$ be a morphism of bitopological
  spaces. Since $f^{-1}(\cT) \subseteq \tau_+$ and $f^{-1}(\cT^c)
  \subseteq \tau_-$, the underlying set function of $f$ defines a
  morphism $g: \clf_+(\cX) \to (Y, \cT)$. Clearly, $g$ is the unique
  morphism satisfying $\psk(g) \circ \eta_{\cX} = f$, and this proves
  our claim.
\end{proof}

We now describe the equivalence of categories determined by $\clf_+
\dashv \psk$.

\begin{proposition}\label{biadjp}
  The fixpoints of the adjunction $\clf_+: \bitop\lra\perv:\psk$ are,
  respectively, the zero-dimensional bispaces and the strongly exact
  Pervin spaces.
\end{proposition}
\begin{proof}
  It follows from the proof of Lemma~\ref{biadj} that the unit of the
  adjunction $\clf_+ \dashv \psk$ at a bispace $\cX = (X, \,\tau_+,\,
  \tau_-)$ is the morphism
  \[\eta_{\cX}: (X, \, \tau_+, \,\tau_-) \to (X,
    \,\Omega_{\cl_+(\cX)}(X),\, \Omega_{\cl_-(\cX)}(X))\]
  defined by the identity map on~$X$.  Thus, $\cX$ is a fixpoint of
  the adjunction if and only if $\tau_+ = \Omega_{\cl_+(\cX)}(X)$ and
  $\tau_- = \Omega_{\cl_-(\cX)}(X)$, that is, if and only if $\cX$ is
  zero-dimensional.

  Let us now exhibit the counit of $\clf_+ \dashv \psk$. It is easy to
  see that the positive clopens of $\psk(X, \cS) =
  (X,\,\Omega_\cS(X),\, \Omega_{\cS^c}(X))$ are the open intersections
  of~$\cS$.  Since $\cS \subseteq \sei \cS$, the identity on~$X$
  defines a morphism
  \[\varepsilon_{(X, \cS)}:(X, \sei {\cS}) \to (X, \cS)\]
  of Pervin spaces, which can be shown to be the counit of the
  adjunction. In particular, we have that $(X, \cS)$ is a fixpoint if
  and only if $\cS = \sei \cS$, that is, if and only if $(X, \cS)$ is
  a strongly exact Pervin space.
\end{proof}
\begin{corollary}\label{pointsetworks}
  The categories $\bitop_{\rm Z}$ and $\perv_{se}$ of zero-dimensional
  bitopological spaces and of strongly exact Pervin spaces are
  equivalent.
\end{corollary}
We finally remark that, since, for every Pervin space $(X, \cS)$, we
have
\begin{align*}
  \psk \circ \clf_+ \circ \psk(X, \cS)
  & = \psk(X, \sei\cS)
    = (X,
    \Omega_{\sei \cS}(X), \Omega_{(\sei \cS)^c}(X))
  \\ & = \psk(X, \cS),
\end{align*}
the unit $\eta_{\psk(X, \cS)}$ is always an isomorphism, and thus, the
adjunction $\clf_+ \dashv \psk$ is idempotent.
 
Let us now look at the pointfree version of the Skula functor and its
left adjoint.
For a Frith frame $(L,S)$ we set $\fsk(L,S)=(\cC_SL, \nabla L, \Delta
S)$.\footnote{As for spaces, this is the underlying biframe of the
  quasi-uniform frame defined by~$(L, S)$.} In order to define $\fsk$
on morphisms, we first observe that, by Proposition~\ref{p:10}, every
morphism of Frith frames $h:(L,S)\ra (M,T)$ uniquely extends to a
frame homomorphism $\overline h: \cC_SL \to \cC_TM$ satisfying
\[\overline h[\nabla L] = \nabla h[L] \qquad \text{and}
  \qquad\overline{h}[\Delta S] = \Delta h[S].\]
Therefore, $\overline{h}$ defines a biframe homomorphism $\overline h:
\fsk(L, S) \to \fsk(M, T)$ and we may set $\fsk(h) = \overline h$.

In the other direction, we define $\bbf_+: \bifrm \to \ffrm$ as follows. For a
biframe $\cL = (L, L_+, L_-)$, we set $\bbf_+(\cL) = (\langle
\bb_+(\cL)\rangle_\Frm,\, \bb_+(\cL))$, where $\langle
\bb_+(\cL)\rangle_\Frm$ denotes the subframe of $L_+$ generated by the
lattice $\bb_+(\cL)$ of positive bicomplemented elements of~$\cL$. In
order to define $\bbf_+$ on morphisms, notice that, if $h: \cL \to
\cK$ is a biframe homomorphism then, since positive bicomplemented
elements of $\cL$ are mapped to positive bicomplemented elements of
$\cK$, the suitable restriction and co-restriction of~$h$ induces a
morphism of Frith frames $\bbf_+h: \bbf_+(\cL) \to \bbf_+(\cK) $.

Next we will see that, as for Pervin spaces, the functors $\fsk$ and
$\bbf_+$ define an adjunction between the categories of Frith frames
and of biframes. However, while the fixpoints of $\ffrm$ are still
easy to describe, the same does not happen with those of $\bifrm$. We
leave it as an open problem to describe the categorical equivalence
underlying this adjunction.

Before proceeding, we prove the following technical result:
\begin{lemma}\label{l:1}
  Let $(L, S)$ be a Frith frame and $a \in L$. Then, $\nabla_a$ is a
  positive bicomplemented element of $(\cC_SL, \nabla L, \Delta S)$ if
  and only if $a$ is a strongly exact meet of elements of~$S$.
\end{lemma}
\begin{proof}
  Since $\cC_SL$ is a subframe of $\cC L$, we have that $\nabla_a$ is
  bicomplemented if and only if $\Delta_a \in \Delta S$, that is, if
  and only if there is some $P \subseteq S$ such that $\Delta_a =
  \bigvee_{s \in P} \Delta_s$. But this is the same as saying that
  $\bigwedge P$ is a strongly exact meet and $a = \bigwedge P$.
\end{proof}

We are now able to prove that $\fsk$ is indeed the left adjoint
of~$\bbf_+$.
\begin{lemma}\label{l:3}
  The functor $\fsk$ is left adjoint to $\bbf_+: \bifrm \to \ffrm$.
\end{lemma}
\begin{proof}
  It follows from Lemma~\ref{l:1} that $\bbf_+\circ \fsk(L, S) =
  (\nabla L, \{\nabla_a \mid a \in \sem S\})$ and thus, the
  isomorphism $\nabla: L \to \nabla L$ induces an embedding of Frith
  frames
  \[\eta_{(L, S)}: (L, S) \to \bbf_+\circ \fsk(L, S).\]
  To conclude that $\fsk \dashv \bbf_+$, it suffices to show that
  $\eta_{(L,S)}$ is universal from $(L, S)$ to $\bbf_+$, that is, that
  for every biframe $\cK$ and every morphism $h: (L, S) \to
  \bbf_+(\cK)$, there exists a unique $h': \fsk(L, S) \to \cK$ such
  that $\bbf_+(h') \circ \eta_{(L, S)} = h$. But the underlying frame
  homomorphism of such an $h'$ has to be an extension $h': \cC_SL \to
  K$ of~$h$. Since $h[S]$ consists of complemented elements of~$K$, by
  Proposition~\ref{p:10}, there exists exactly one such morphism, which is
  easily seen to define a biframe homomorphism $h': \fsk(L, S) \to
  \cK$.
\end{proof}

\begin{corollary}
  \label{c:2}
  The fixpoints of $\ffrm$ for the adjunction $\fsk \dashv \bbf_+$ are
  the strongly exact Frith frames.
\end{corollary}
\begin{proof}
  It follows from the proof of Lemma~\ref{l:3} that the unit of the
  adjunction $\fsk \dashv \bbf_+$ is
  \[\eta_{(L, S)}: (L, S) \to (\nabla L, \{\nabla_a \mid a \in \sem
    S\}), \quad  a \mapsto \nabla_a.\]
  Clearly, this is an isomorphism if and only if $(L, S)$ is strongly
  exact.
\end{proof}

Let $\cL = (L, L_+, L_-)$ be a biframe and $\bbf_+(\cL) = (M,
T)$. Since the elements of $T$ are complemented in $L$, by
Proposition~\ref{p:10}, the frame embedding $M \hookrightarrow L$ may
be uniquely extended to a frame homomorphism $\cC_TM \to L$. It is
easily seen that this map induces a biframe homomorphism
$\varepsilon_\cL: (\cC_TM, \nabla M, \Delta T) \to (L, L_+, L_-)$. We
leave it for the reader to verify that $\varepsilon_\cL$ is the
component at~$\cL$ of the counit of the adjunction $\fsk \dashv
\bbf_+$. In particular, if $(L, S)$ is a Frith frame then, since
$\{\Delta_s \mid s \in S\}$ and $\{\Delta_a \mid a \in \sem S\}$
generate the same subframe of $\cC L$, by Lemma~\ref{l:1}, we have
that $\fsk \circ \bbf_+ \circ \fsk(L, S) = \fsk(L, S)$ and
$\varepsilon_{\fsk(L,S)}$ is the identity map.  Therefore, the
adjunction $\fsk \dashv \bbf_+$ is idempotent and, as such, it induces
an equivalence between the images of the two involved functors.

While it is clear that every biframe of the form $\fsk(L, S)$ is
zero-dimensional, it is not the case that every zero-dimensional
biframe is of that form.

\begin{example}\label{sec:1}\footnote{This example was borrowed from~\cite[Example
    5.13]{borlido21}. The interested reader may show that the
    fixpoints of the adjunction $\fsk \dashv \bbf_+$ are precisely the
    underlying biframes of the quasi-uniform frames representable by a
    Frith frame in the sense of~\cite[Proposition~5.6]{borlido21}.}
  Let $X$ be a topological space such that the congruence frame of its
  frame of opens is not spatial (see~\cite[Theorem~3.4]{niefield87}
  for a characterization of the frames whose congruence frame is not
  spatial). We let $\cL = (L, L_+, L_-)$ be the \emph{Skula biframe
    of~$X$}, that is: $L_+$ is the frame of opens of~$X$, $L_-$ is the
  subframe of $\cP(X)$ generated by the complements of the elements of
  $L_+$, and $L$ is the subframe of $\cP(X)$ generated by $L_+ \cup
  L_-$. To show that $\cL$ is not a fixpoint of the adjunction $\fsk
  \dashv \bbf_+$, we first recall that the underlying frame
  homomorphism of the counit of $\fsk \dashv \bbf_+$ at a biframe $\cL
  = (L, L_+, L_-)$ is the unique frame extension $\varepsilon_{\cL}:
  \cC_T M \to L$ of the embedding $M \hookrightarrow L$, where $(M, T)
  = \bbf_+(\cL)$. Since, in this case, we have $\bbf_+(\cL) = (L_+,
  L_+)$, $\cL$ is a fixpoint if and only if the unique frame
  homomorphism $\cC L_+ \to L$ extending $L_+ \hookrightarrow L$ is an
  isomorphism. But that is not the case because $L$ is spatial and
  $\cC L_+$ is not.
  A concrete example is given by taking for $X$ the real line $\mathbb
  R$ equipped with the Euclidean topology $\Omega(\mathbb R)$. Since
  the Booleanization of $\Omega(\mathbb R)$ is a pointless nontrivial
  sublocale, by the characterization of~\cite{niefield87}, its
  congruence frame is not spatial.
\end{example}

Also, the fixpoints of $\fsk \dashv \bbf_+$ need not be compact as
$\fsk(L, S)$ is not compact if neither is $L$. We may however show
that every compact and zero-dimensional biframe is a fixpoint.

\begin{proposition}
  \label{p:8}
  Let $\cL = (L, L_+, L_-)$ be a biframe and $\varepsilon_\cL:
  (\cC_TM, \nabla M, \Delta T) \to (L, L_+, L_-)$ be the component
  at~$\cL$ of the counit of the adjunction $\fsk \dashv \bbf_+$, where
  $(M, T) = \bbf_+(\cL)$. Then,
  \begin{enumerate}
  \item $\varepsilon_\cL$ is dense,
  \item\label{item:13} if $\cL$ is zero-dimensional, then
    $\varepsilon_\cL[\nabla M] = L_+$ and $\varepsilon_\cL[\Delta T] =
    L_-$.
  \end{enumerate}
  In particular, if $\cL$ is compact and zero-dimensional, then
  $\varepsilon_\cL$ is an isomorphism and thus, $\cL$ is a fixpoint of
  the adjunction $\fsk\dashv \bbf_+$.
\end{proposition}
\begin{proof}
  Let $a \in M$ and $t \in T$ be such that $\varepsilon_\cL(\nabla_a
  \wedge \Delta_t) = 0$. By definition of~$\varepsilon_\cL$, this is
  the same as having that the equality $a \wedge t^* = 0$ holds in
  $L$. Since $t$ is complemented in $L$, this is equivalent to $a \leq
  t$ which, in turn, implies $\nabla_a \wedge \Delta_t = 0$. This
  proves that $\varepsilon_\cL$ is dense.

  Now, again by definition of $\varepsilon_\cL$, we have that
  $\varepsilon_\cL[\nabla M] \supseteq T$ and $\varepsilon_\cL[\Delta
  T] \supseteq T^*$. Since $T$ is the lattice of bicomplemented
  elements of $L_+$, if $\cL$ is zero-dimensional, this implies that
  $\varepsilon_\cL[\nabla M] = L_+$ and $\varepsilon_\cL[\Delta T] =
  L_-$. Thus, \ref{item:13} holds.

  Finally, recall that $\varepsilon_\cL$ is an isomorphism of biframes
  provided its underlying frame homomorphism is injective and
  satisfies $\varepsilon_\cL[\nabla M] = L_+$ and
  $\varepsilon_\cL[\Delta T] = L_-$. Thus, it suffices to show that if
  $\cL$ is compact and zero-dimensional then $\varepsilon_\cL$ is
  injective. But it is well-known that dense frame homomorphisms with
  zero-dimensional domain (which is the case of $\cC_TM$) and compact
  codomain are injective (see e.g.~\cite[Chapter~VII,
  Proposition~2.2.2]{picadopultr2011frames}).\footnote{The result
    cited is stated for a \emph{regular} domain, but every
    zero-dimensional frame is \emph{regular}.}
\end{proof}

The following is as close as we will get to a pointfree version of the
result stated in Corollary~\ref{pointsetworks}.

\begin{corollary}
  Strongly exact Frith frames are a full coreflective subcategory of
  the category of zero-dimensional biframes.
\end{corollary}

We finish this section by relating the point-set and pointfree versions
of the functors we have considered.
\begin{proposition}\label{p:14}
  The following squares commute up to natural isomorphism.
  \begin{center}
    \begin{tikzpicture}
      \node (A) {$\bitop$}; \node[right of = A, xshift = 20mm] (B)
      {$\perv$}; \node[below of = A] (C) {$\bifrm$}; \node[below of =
      B] (D) {$\ffrm$};
      \draw (A) to node[ArrowNode,above] {$\clf_+$} (B); \draw (B) to
      node[ArrowNode,right] {$\bf \Omega$} (D); \draw (A) to
      node[ArrowNode,left] {${\bf \Omega}_b$} (C); \draw (C) to
      node[ArrowNode, below] {$\bbf_+$} (D);
      \node[right of = A, xshift = 50mm] (A') {$\perv$}; \node[right
      of = A', xshift = 20mm] (B') {$\bitop$}; \node[below of = A']
      (C') {$\ffrm$}; \node[below of = B'] (D') {$\bifrm$};
      \draw (A') to node[ArrowNode,above] {$\psk$} (B'); \draw (D') to
      node[ArrowNode,right] {$\ptf_b$} (B'); \draw (C') to
      node[ArrowNode,left] {$\ptf$} (A'); \draw (C') to
      node[ArrowNode, below] {$\fsk$} (D');
    \end{tikzpicture}
  \end{center}
\end{proposition}
\begin{proof}
  Commutativity of the left-hand side diagram follows easily from
  computing the functors $\bbf_+ \circ {\bf \Omega}_b$ and
  ${\bf\Omega} \circ \clf_+$.
  To show that the right-hand side diagram commutes up to natural
  isomorphism, we define a natural isomorphism $\beta: \psk \circ \ptf
  \implies \ptf_b \circ \fsk$ as follows.  For a Frith frame~$(L, S)$,
  we define
  \[\beta_{(L,S)}:\psk \circ \ptf(L,S) \ra \ptf_b\circ \fsk(L,S),\qquad
    p \mapsto \widetilde{p}\]
  where $\widetilde{p}$ is the unique morphism making the following
  diagram commute (cf. Proposition~\ref{p:10}).
  \begin{center}
    \begin{tikzpicture}
      [->, node distance = 20mm] \node (A) {$L$}; \node[right of = A]
      (B) {$\cC_S L$}; \node[below of = B, yshift = 5mm] (C) {$\two$};
      \draw[right hook->] (A) to node[above, ArrowNode] {$\nabla$}
      (B); \draw[dashed] (B) to node[right, ArrowNode, yshift = 1mm]
      {$\widetilde p$} (C); \draw (A) to node[left, ArrowNode, xshift
      = 0mm, yshift = -1mm] {$p$}(C);
    \end{tikzpicture}
  \end{center}
  By uniqueness of each $\widetilde p$, this correspondence
  establishes a bijection between the points of~$L$ and the points
  of~$\cC_SL$. Let us show that this is a homeomorphism with respect
  to the first topology of $\psk \circ \ptf(L,S)$. We first note that
  the positive open subsets of $\psk \circ \ptf(L,S)$ are the subsets
  of the form $\widehat a$, while those of $\ptf_b\circ \fsk(L,S)$ are
  the subsets of the form $\widehat{\nabla_a}$, where $a \in L$. Now,
  given $a \in L$ and $p \in \pt(L)$, using commutativity of the
  triangle above, we have
  \[p \in \beta_{(L, S)}^{-1}(\widehat{\nabla_a}) \iff \widetilde
    p(\nabla_a) = 1 \iff p(a) = 1 \iff p \in \widehat a.\]
  Since $\beta_{(L, S)}$ is a bijection, this implies that $\beta_{(L,
    S)}$ is both continuous and open with respect to the first
  topologies, thus a homeomorphism. Showing that $\beta_{(L, S)}$ is
  also a homeomorphism with respect to the second topology is
  analogous, and we leave it for the reader.

  Now, $\beta$ is a natural transformation provided the following
  square commutes for every morphism $h: (M,T) \to (L, S)$ of Frith
  frames.
  \begin{center}
    \begin{tikzpicture}
      \node (A) {$\psk \circ \ptf(L, S)$}; \node[right of = A, xshift
      = 40mm] (B) {$\bpt\circ \fsk(L,S)$}; \node[below of = A] (C)
      {$\psk \circ \ptf(M,T)$}; \node[below of = B] (D) {$\bpt \circ
        \fsk(M,T)$};
      \draw (A) to node[ArrowNode,above] {$\beta_{(L,S)}$} (B); \draw (B) to
      node[ArrowNode,right] {$(-)\circ \fsk(h)$} (D); \draw (A) to
      node[ArrowNode,left] {$(-)\circ h$} (C); \draw (C) to node[ArrowNode,
      below] {$\beta_{(M,T)}$} (D);
    \end{tikzpicture}
  \end{center}
  That is indeed the case because, for every $p\in \pt(L)$ and $x \in
  M$, we have the following equalities:
  \[\widetilde{p}\circ \fsk(h)(\nabla_{x})
    =\widetilde{p}(\nabla_{h(x)})= p\circ h(x) =\widetilde{p\circ
      h}(\nabla_{x}).\popQED\]
\end{proof}

Proposition~\ref{p:14} makes it natural to ask whether the
corresponding diagrams for the spectrum and open-set functors also
commute. The answer is negative, as shown by the next example.

\begin{example}\label{sec:2}
  For the first diagram, consider the biframe $\cL = ({\bf 3}, {\bf
    3}, \two)$, where ${\bf 3}$ denotes the $3$-element chain. Then,
  $\ptf\circ \bbf_+ (\cL)$ is a space with one point, while
  $\clf_+\circ \bpt (\cL)$ has two, so these cannot be isomorphic.
  For the second diagram, we let $X$ be a topological space as in
  Example~\ref{sec:1} and observe that ${\bf \Omega }_b\circ \psk(X,
  \Omega(X))$ is the \emph{Skula biframe} $\cL$ of $X$. As already
  argued, $\cL$ is not in the image of $\fsk$, thus, ${\bf \Omega}_b
  \circ \psk(X, \Omega(X))$ is not isomorphic to $\fsk \circ {\bf
    \Omega}(X, \Omega(X))$.
\end{example}

\subsection{The monotopological case}\label{sec:zero-dim}
In this section, we will investigate the monotopological version of
the results of Section~\ref{sec:se}. Under the identifications
$\Top \hookrightarrow \bitop$ and $\Frm \hookrightarrow \bifrm$, these
will follow as a consequence of the latter.

Let us consider the restrictions $\clf: \Top \to \perv$ and $\bbf: \Frm
\to \ffrm$ of the functors $\clf_+$ and $\bbf_+$ defined in
Section~\ref{sec:se}. Explicitly, $\clf$ maps a topological space
$(X, \tau)$ to the Pervin space $\clf_+(X, \tau, \tau) = (X,\, \cl(X,
\tau))$, where $\cl(X, \tau)$ denotes the Boolean algebra of clopen
subsets of $X$ for the topology $\tau$, and a morphism to itself. On
the other hand, given a frame $L$, $\bbf(L)$ is the pair $(\langle
\bb(L)\rangle_\Frm, \,\bb(L))$, where $\langle \bb(L)\rangle_\Frm$
denotes the subframe of $L$ generated by the lattice of complemented
elements~$\bb(L)$ of~$L$, and a frame homomorphism $h:L \to K$ is sent
to the morphism of Frith frames $\bbf h: \bbf(L) \to \bbf(K)$ induced by
the suitable restriction and co-restriction of~$h$.
Then, the adjunctions $\clf_+: \bitop\lra\perv:\psk$ and $\fsk: \ffrm
\lra \bifrm: \bbf_+$ studied in Section~\ref{sec:se} restrict,
respectively, to adjunctions $\clf: \Top\lra\perv':\uperv'$ and
$\ufrith': \ffrm' \lra \Frm: \bbf$, where
\begin{itemize}
\item $\perv'$ denotes the full subcategory of $\perv$ determined by
  the Pervin spaces $(X, \cS)$ such that $\psk(X, \cS)$ belongs to the
  image of $\Top \hookrightarrow \bitop$,
\item $\ffrm'$ denotes the full subcategory of $\ffrm$ determined by
  the Frith frames $(L, S)$ such that $\fsk(L, S)$ belongs to the
  image of $\Frm \hookrightarrow \bifrm$,
\item $\uperv'$ is the suitable restriction and co-restriction of
  $\psk$, and
\item $\ufrith'$ is the suitable restriction and co-restriction of
  $\fsk$.
\end{itemize}
The following result, whose proof is trivial, explains our choice of
notation for the functors $\uperv'$ and $\ufrith'$: these are nothing
but the suitable restrictions of the functors $\uperv$ and $\ufrith$
defined in Section~\ref{sec:forget}.

\begin{lemma}\label{l:15} The following statements hold:
  \begin{enumerate}
  \item\label{item:11} a Pervin space $(X, \cS)$ belongs to $\perv'$
    if and only if $\Omega_\cS(X) = \Omega_{\cS^c}(X)$,
  \item\label{item:12} a Frith frame $(L, S)$ belongs to $\ffrm'$ if
    and only if $\nabla L = \Delta S$.
  \end{enumerate}
  In particular, given $(X, \cS) \in \perv'$ and $(L, S) \in \ffrm'$,
  the following equalities hold:
  \[\uperv'(X, \cS) = (X,\, \Omega_\cS(X))\qquad \text{ and
    }\qquad\ufrith'(L, S) = L.\]
\end{lemma}

We will now characterize the categorical equivalences induced by the
adjunctions $\clf \dashv \uperv'$ and $\ufrith' \dashv \bbf$. Recall
that we have seen in Section~\ref{sec:se} that both $\clf_+ \dashv
\psk$ and $\fsk \dashv \bbf_+$ are idempotent and, therefore, so are
their restrictions.

\begin{proposition}\label{p:2}
  The categories of zero-dimensional topological spaces and that of
  strongly exact symmetric Pervin spaces are equivalent.
\end{proposition}
\begin{proof}
  Since $\clf \dashv \uperv'$ is a restriction of $\clf_+ \dashv
  \psk$, by Proposition~\ref{biadjp}, it induces an equivalence
  between the categories of topological spaces that are
  zero-dimensional when seen as bitopological spaces, and the category
  determined by the Pervin spaces $(X, \cS) \in \perv'$ that are
  strongly exact. The former are easily seen to be the
  zero-dimensional topological spaces. We argue that $(X, \cS) \in
  \perv'$ is strongly exact if and only if it is a strongly exact
  symmetric Pervin space. Clearly, $\perv'$ contains all symmetric
  Pervin spaces, thus the backwards implication is
  trivial. Conversely, if $(X, \cS) \in \perv'$ is strongly exact
  then, being a fixpoint of the idempotent adjunction $\clf \dashv
  \uperv'$, it belongs to the image of $\clf$. Hence, it is symmetric,
  as required.
\end{proof}

\begin{proposition}
  The categories of zero-dimensional frames and that of strongly exact
  symmetric Frith frames are equivalent.
\end{proposition}
\begin{proof}
  Since the adjunction $\ufrith' \dashv \bbf$ is idempotent, its
  fixpoints in $\Frm$ are the frames of the form $\ufrith'(L, S) = L$,
  for $(L, S) \in \ffrm'$. Noticing that the inclusion $\nabla L
  \supseteq \Delta S$ holds if and only if $S$ consists of
  complemented elements, these are easily seen to be the
  zero-dimensional ones. On the other hand, an argument similar to
  that used in the proof of Proposition~\ref{p:2} shows that the
  fixpoints of $\ufrith' \dashv \bbf$ in $\ffrm'$ are strongly exact
  symmetric Frith frames.
\end{proof}


\section{Complete Pervin spaces and complete Frith
  frames}\label{sec:completion}
In this section we will show that the dual adjunction ${\bf \Omega}:
\perv \lra \ffrm : \ptf$ induces a duality between $T_0$ complete
Pervin spaces on the one hand and complete Frith frames on the
other. We use the following definition
from~\cite{GehrkeGrigorieffPin2010, pin17}.

\begin{definition}[{\cite{GehrkeGrigorieffPin2010,pin17}}]\label{sec:ccomplete}
  Let $(X, \cS)$ be a Pervin space. A filter $F \subseteq \cP(X)$ is a
  \emph{Cauchy filter} if it is proper and, for every $S\in \ca{S}$,
  either $S$ or its complement is in $F$. We say that a Cauchy filter
  $F$ \emph{converges} to the point $x \in X$ if every open
  neighborhood $U \in \Omega_{\overline{\cS}}(X)$ of~$x$ belongs to
  $F$. Finally, a Pervin space $(X,\ca{S})$ is said to be \emph{Cauchy
    complete} if every Cauchy filter converges, and a \emph{Cauchy
    completion} of $(X, \cS)$ is a dense extremal monomorphism $c: (X,
  \cS) \hookrightarrow (Y, \cT)$ into a Cauchy complete Pervin space
  $(Y, \cT)$.\footnote{In~\cite{GehrkeGrigorieffPin2010} \emph{Cauchy
      complete} and \emph{Cauchy completion} are simply named
    \emph{complete} and \emph{completion}, respectively.}
\end{definition}
In the following, we refer to the symmetrization of a Pervin space, defined in subsection 2.5. The following is an easy observation that we state for later
reference.
\begin{lemma}\label{l:7}
  Let $(X, \cS)$ be a Pervin space, and $F \subseteq \cP(X)$ be a
  Cauchy filter. Then, $F$ converges to $x$ if and only if $x$ belongs
  to $\bigcap (F \cap \overline{\cS})$.
\end{lemma}

Note that a filter is Cauchy with respect to $(X, \cS)$ if and only if
it is Cauchy with respect to $(X, \overline{\cS})$. Therefore, a
Pervin space is Cauchy complete if and only if so is its
symmetrization.  As observed in~\cite{GehrkeGrigorieffPin2010,pin17},
one may show that this notion of Cauchy complete Pervin space
correctly captures the notion of a complete quasi-uniform space.
It is known that complete quasi-uniform spaces may be equivalently
characterized via dense extremal monomorphisms.  In the case of Pervin
spaces, the suitable definitions are the following.

\begin{definition}\label{sec:complete}
  A symmetric Pervin space $(X, \cB)$ is \emph{complete} if every
  dense extremal monomorphism $(X, \cB) \hookrightarrow (Y, \cC)$,
  with $(Y, \cC)$ a $T_0$ symmetric Pervin space is an isomorphism.
  More generally, we say that a Pervin space $(X, \cS)$ is
  \emph{complete} if so is its symmetrization.
\end{definition}

Our next goal is to show that Definitions~\ref{sec:ccomplete}
and~\ref{sec:complete} are equivalent.  Before we move on, we need to
prove a couple of technical lemmas.

\begin{lemma}\label{tech1}
  Let $m:(X,\cS)\hookrightarrow (Y,\cT)$ be a dense extremal
  monomorphism. If $F \subseteq \cP(Y)$ is a Cauchy filter, then so is
  $m^{-1}(F)$.
\end{lemma}
\begin{proof}
  Let $F\se \cP(Y)$ be a Cauchy filter. Since $m$ is dense and $F$ is
  proper, $m^{-1}(F)$ is, by Lemma~\ref{l:9}, a proper filter
  too. Now, given $S\in \cS$, since $m$ is an extremal monomorphism,
  we have $S=m^{-1}(T)$ for some $T\in \cT$. Since $F$ is Cauchy, it
  contains either $T$ or $T^c$ and thus, $m^{-1}(F)$ contains either
  $S = m^{-1}(T)$ or $S^c = m^{-1}(T^c)$. This shows that $m^{-1}(F)$
  is Cauchy as well.
\end{proof}

Recall the neighborhood map $\cN_{(X, \cS)}:(X,\cS)\to (\pf(\cS),
\,\widetilde \cS)$ from~\eqref{eq:3}, that is, the unit of the
adjunction $\lperv \dashv \pff$.
\begin{lemma}\label{tech3}
  For a $T_0$ Pervin space $(X,\cS)$, the map $\cN_{(X, \cS)}$ is an
  extremal monomorphism of Pervin spaces whose symmetrization is
  dense.
\end{lemma}
\begin{proof} If $(X,\cS)$ is $T_0$, then different points have
  different neighborhood filters in $\cS$, and so $\cN_{(X,\cS)}$ is
  injective.  Since, for every $S \in \cS$, we have $\cN_{(X,
    \cS)}^{-1}(\widetilde S) = S$, the map $\cN_{(X, \cS)}$ is an
  extremal monomorphism.  To show that $\psym(\cN_{(X, \cS)})$ is
  dense, suppose that $\cN_{(X,\cS)}^{-1}(\widetilde{S_1}\cap
  \widetilde{S_2}^c) = S_1 \cap S_2^c=\emptyset$, that is, $S_1\se
  S_2$. Then, there is no prime filter containing~$S_1$ and
  omitting~$S_2$, which means that $\widetilde{S_1}\cap
  \widetilde{S_2}^c$ must be empty.
\end{proof}

We remark that, for every $T_0$ Pervin space $(X, \cS)$, the map
$\cN_{(X,\cS)}: (X, \cS) \hookrightarrow (\pf(\cS), \widetilde \cS)$
is the completion of~$(X, \cS)$ (cf.~\cite{GehrkeGrigorieffPin2010,
  pin17}).

We may now prove the following characterization of $T_0$ complete
Pervin spaces.

\begin{theorem}\label{char}
  Let $(X, \cS)$ be a $T_0$ Pervin space. Then, the following are
  equivalent:
  \begin{enumerate}
  \item \label{char1} $(X,\cS)$ is Cauchy complete;
  \item \label{char2} $(X,\cS)$  is complete;
  \item \label{char3} Every extremal monomorphism
    $(X,\cS)\hookrightarrow (Y,\cT)$ into a $T_0$ Pervin space whose
    symmetrization is dense is an isomorphism;
  \item \label{char4} $(X,\cS)$ is isomorphic to $\pff(\cS)$;
  \item \label{item:2} $(X,\cS)$ is isomorphic to $\ptf(\idl(\cS),
    \cS)$;
  \item \label{char6} $(X, \cS)$ is isomorphic to a Pervin space of
    the form $\ptf(\idl(D), D)$, for some lattice~$D$;
  \item \label{char5} $(X, \cS)$ is isomorphic to a Pervin space of
    the form $\pff(D)$, for some lattice~$D$.
  \end{enumerate}
\end{theorem}
\begin{proof}
  Noting that extremal monomorphisms are preserved under
  symmetrization, the equivalence between~\ref{char2} and~\ref{char3}
  follows. That~\ref{char3} implies~\ref{char4} is a consequence of
  Lemma~\ref{tech3}. The equivalences between~\ref{char4}
  and~\ref{item:2} and between~\ref{char6} and~\ref{char5} follow from
  Lemma~\ref{l:8}, while that~\ref{item:2} implies~\ref{char6} is
  trivial.  It remains to show that~\ref{char1} implies~\ref{char2}
  and that~\ref{char5} implies~\ref{char1}.

  Suppose that~\ref{char1} holds, and let
  $m:(X,\overline{\cS})\hookrightarrow (Y,\cC)$ be a dense extremal
  monomorphism into a symmetric $T_0$ Pervin space $(Y, \cC)$. We need
  to show that~$m$ is an isomorphism, that is, that~$m$ is
  surjective. Given $y \in Y$, consider the filter $F_y := \up\{C \in
  \cC \mid y \in C\}$. Since $\cC$ is a Boolean algebra, $F_y$ is a
  Cauchy filter. By Lemma~\ref{tech1}, the filter $m^{-1}(F_y)$ is
  Cauchy as well. Since $(X, \cS)$ is Cauchy complete, $m^{-1}(F_y)$
  converges to some point $x \in X$. We claim that $y = m(x)$. Since
  $(Y,\cC)$ is $T_0$ and $\cC$ is a Boolean algebra, we have that $y =
  m(x)$ provided $y \in C$ implies $m(x) \in C$, for every $C \in
  \cC$. We let $C \in \cC$ be such that $y \in C$. Then, $m^{-1}(C)$
  belongs to $m^{-1}(F_y) \cap \overline{\cS}$ and, since
  $m^{-1}(F_y)$ converges to $x$, by Lemma~\ref{l:7}, it follows that
  $x \in m^{-1}(C)$, that is, $m(x) \in C$, as required.

  Finally, let us assume that we have a Pervin space of the form
  $\pff(D) = (\pf(D), \widetilde D)$, with $D$ a distributive
  lattice. Suppose that $F\se \cP(\pf(D))$ is a Cauchy filter and set
  $P:= \{a \in D \mid \widetilde a \in F\}$. Clearly, $P$ is a filter
  of~$D$. Let us show that $P$ is prime. Let $a,b \in D$ be such that
  $a \vee b \in P$, and suppose that $a \notin P$. Equivalently, $a,b$
  are such that $\widetilde {a \vee b} = \widetilde a \cup \widetilde
  b \in F$ and $\widetilde a \notin F$. Since $F$ is Cauchy, it
  follows that $(\widetilde a)^c$ belongs to~$F$ and thus, so does
  $\widetilde b \supseteq (\widetilde a)^c \cap \widetilde b =
  (\widetilde a)^c \cap(\widetilde a \cup \widetilde b)$. This shows
  that $b \in P$ as required. We now claim that $F$ converges to
  $P$. By Lemma~\ref{l:7}, it suffices to show that $P \in \widetilde
  a \cap (\widetilde b)^c$ whenever $a,b \in D$ are such that
  $\widetilde a \cap (\widetilde b)^c \in F$. Since $F$ is proper,
  having $\widetilde a \cap (\widetilde b)^c \in F$ implies that
  $\widetilde a \in F$ and $\widetilde b \notin F$. But by definition
  of~$P$, this means that $P \in \widetilde a \cap (\widetilde b)^c$,
  as required.
\end{proof}

From now on, we will drop the use of \emph{Cauchy complete} and call
\emph{complete Pervin space} every Pervin space satisfying the
equivalent conditions of Theorem~\ref{char}.

In turn, completeness of Frith frames is discussed
in~\cite{borlido21}, where a characterization of complete Frith frames
using both dense extremal epimorphisms and Cauchy maps is
given. Here, it will suffice to consider the following definitions:
\begin{definition}[{\cite{borlido21}}]
  We say that a symmetric Frith frame $(L, B)$ is \emph{complete} if
  every dense extremal epimorphism $(M, C) \twoheadrightarrow (L, B)$
  with $(M, C)$ symmetric is an isomorphism. More generally, a Frith
  frame $(L, S)$ is \emph{complete} provided its symmetric
  reflection~$\fsym(L, S)$ is complete.  A \emph{completion} of $(L,
  S)$ is a complete Frith frame $(M, T)$ together with a dense
  extremal epimorphism $(M, T) \twoheadrightarrow (L, S)$.
\end{definition}

The fact that completeness of a Frith frame $(L, S)$ is equivalent to
completeness of the associated quasi-uniform frame is shown
in~\cite[Proposition~7.2]{borlido21}. Moreover, every Frith frame $(L,
S)$ has a unique, up to isomorphism, completion, which is given by the
counit~\eqref{eq:5} of the adjunction $\idlf \dashv \lfrith$. Also
in~\cite{borlido21}, we have shown the following:

\begin{theorem}
  [{\cite[{Proposition 4.6 and
      Theorem~7.7}]{borlido21}}]\label{cfrith} Let $(L, S)$ be a Frith
  frame. Then, the following are equivalent:
  \begin{enumerate}
  \item $(L, S)$ is complete;
  \item $(L, S)$ is coherent;
  \item\label{item:4} $L$ is isomorphic to the ideal completion
    $\idl(S)$ of $S$.
  \end{enumerate}
\end{theorem}

Part~\ref{item:4} of this characterization, together with
Theorem~\ref{char}\ref{char6}, yield the following:

\begin{corollary}\label{restrict1}
  A Pervin space is $T_0$ and complete if and only if it is of the
  form $\ptf(L, S)$, for some complete Frith frame $(L, S)$.
\end{corollary}

In particular, the functor $\ptf: \ffrm\to \perv$ restricts and
co-restricts to a functor $\cffrm \to \cperv$.  In order to show that
$\bf \Omega$, too, restricts correctly, we will need to use the Prime
Ideal Theorem.

\begin{proposition}\label{pit2}
  The following are equivalent.
  \begin{enumerate} 
  \item \label{pit2a}The Prime Ideal Theorem holds;
  \item \label{pit2c}If $(X,\cS)$ is a $T_0$ complete Pervin space,
    then the Frith frame ${\bf\Omega}(X, \cS)$ is complete.
  \item \label{item:3} If $(X,\cB)$ is a $T_0$ complete symmetric
    Pervin space, then the Frith frame ${\bf\Omega}(X, \cB)$ is
    complete.
  \end{enumerate}
\end{proposition}
\begin{proof}
  We first show that~\ref{pit2a} implies~\ref{pit2c}. Let $(X, \cS)$
  be a $T_0$ complete Pervin space. By Theorem~\ref{char}, we may
  assume, without loss of generality, that $(X, \cS) = \ptf(\idl(\cS),
  \cS)$. Since we are assuming that the Prime Ideal Theorem holds, by
  Theorem~\ref{pit}, $\idl(\cS)$ is a spatial frame, and thus so is
  $(\idl(\cS), \cS)$ (cf. Theorem~\ref{dualadj}). Therefore, we have
  \[{\bf \Omega}(X, \cS) = {\bf \Omega }\circ \ptf(\idl(\cS), \cS)
    \cong (\idl(\cS), \cS),\]
  which, by Theorem~\ref{cfrith}, is a complete Frith frame.

  Clearly, \ref{pit2c} implies \ref{item:3}.
  Finally, suppose that~\ref{item:3} holds, and let $X$ be a set. By
  Theorem~\ref{pit}, it suffices to show that every proper filter~$F$
  on~$X$ is contained in a prime filter. For the sake of readability,
  we set $B:= \cP(X)$. Note that showing that $F$ is contained in some
  prime filter is equivalent to showing that the intersection
  $\bigcap_{S \in F} \widetilde S$ is nonempty. For that, we consider
  the $T_0$ symmetric Pervin space $(\pf (B), \widetilde{B})$. By
  Theorem~\ref{char}, this is complete, and by hypothesis, so is the
  Frith frame ${\bf \Omega}(\pf(B),\, \widetilde{B}) =
  (\Omega_{\widetilde{B}}(\pf(B)), \widetilde{B})$. In particular, the
  topological space $(\pf(B), \Omega_{\widetilde{B}}(\pf(B)))$ is
  compact, and thus, it suffices to show that $\{\widetilde S \mid S
  \in F\}$ has the finite intersection property. In turn, since $F$ is
  closed under finite meets, this is the same as showing that
  $\widetilde S \neq \emptyset$ for every $S \in F$.  But since $F$ is
  proper, every $S \in F$ is nonempty, and given $x \in S$, we have
  $\up \{x\}\in \widetilde S$ and $\widetilde S$ is nonempty as
  well. This finishes the proof.
\end{proof}

\begin{corollary}\label{c:4}
  The following are equivalent.
  \begin{enumerate} 
  \item The Prime Ideal Theorem holds;
  \item If $(X,\cS)$ is a $T_0$ complete Pervin space, then the
    compact open subsets of $(X,\, \Omega_\cS(X))$ are precisely the
    elements of $\cS$.
  \end{enumerate}
\end{corollary}
\begin{proof}
  By the equivalence between \ref{pit2a} and~\ref{pit2c} of
  Proposition~\ref{pit2}, it suffices to show that ${\bf\Omega}(X,
  \cS)$ is a complete Frith frame if and only if the compact open
  subsets of $(X,\, \Omega_\cS(X))$ are precisely the elements of
  $\cS$. Observe that, by Theorem~\ref{cfrith}, ${\bf \Omega}(X, \cS)
  = (\Omega_\cS(X), \, \cS)$ is complete if and only if it is
  coherent. Thus, $\cS$ is the set of compact elements of the frame
  $\Omega_\cS(X)$, hence of compact open subsets of the topological
  space $(X, \, \Omega_\cS(X))$.
\end{proof}

We have just proved the following:

\begin{corollary}\label{c:3}
  If the Prime Ideal Theorem holds, then the adjunction ${\bf \Omega}:
  \perv \lra \ffrm: \ptf$ restricts and co-restricts to an adjunction
  between $T_0$ complete Pervin spaces and complete Frith frames.
\end{corollary}

We may now show the main result of this section.

\begin{theorem}\label{t:1}
  If the Prime Ideal Theorem holds, then the adjunction ${\bf
    \Omega}:\perv \lra \ffrm: \ptf$ restricts and co-restricts to a
  duality between the categories $\cperv$ of $T_0$ complete Pervin spaces and the category $\cffrm$ of complete Frith frames.
\end{theorem}
\begin{proof}
  Recall from Theorem~\ref{dualadj} that the adjunction
  ${\bf\Om}\dashv \ptf$ induces a duality between sober Pervin spaces
  and spatial Frith frames.  Thus, because of Corollary~\ref{c:3}, it
  remains to show that $T_0$ complete Pervin spaces are sober and
  complete Frith frames are spatial.  By Theorem~\ref{char}, a $T_0$
  Pervin space is complete if and only if it is isomorphic to the
  Pervin space $\ptf(\idl(D), D)$, for some lattice~$D$, and thus a
  complete Pervin space is sober. In turn, complete Frith frames are
  spatial because, by Theorem~\ref{cfrith}, they are of the form
  $(\idl(D), D)$ for some lattice~$D$ and, by the Prime Ideal Theorem,
  these are spatial Frith frames (cf. Theorems~\ref{pit}
  and~\ref{dualadj}).
\end{proof}


\section{Stone-type dualities}\label{sec:duality}

In this section we will see how several Stone-type dualities relate to
the duality between $T_0$ complete Pervin
spaces and complete Frith frames shown in the previous section (cf. Theorem~\ref{t:1}). Note that
the fact that every $T_0$ complete Pervin space defines a spectral, a
Priestley, and a pairwise Stone space is already observed
in~\cite{pin17}, but no proof is provided. Here, we will give the
functorial details of this assignment, and use Theorem~\ref{t:1} to
interpret Stone-type dualities as a restriction an co-restriction of
the dual adjunction ${\bf \Omega}: \perv \lra \ffrm:\ptf$ along
\emph{full} subcategory embeddings.

\subsection{Stone duality}\label{sec:stone}
Stone duality establishes that the categories of bounded distributive
lattices and of spectral spaces are dually equivalent. We recall that
a topological space is \emph{spectral} if it is sober, and its compact
open subsets are closed under finite intersections and form a basis of
the topology. The category of spectral spaces together with those
continuous functions such that the preimages of compact open subsets
are compact will be denoted by $\spec$.  By identifying each lattice
with the coherent frame given by its ideal completion, Stone duality
may be seen as a restriction and co-restriction of the dual adjunction
${\bf \Omega}: \Top \lra \Frm: \ptf$, but \emph{not along full
  inclusions}, as not every continuous function is a morphism of
spectral spaces, and morphisms of coherent frames are required to
preserve compact elements. We will now see that this duality may also
be seen as a restriction and co-restriction of the dual adjunction
${\bf \Omega}: \perv \lra \ffrm:\ptf$, the advantage being that
spectral spaces and bounded distributive lattices form \emph{full}
subcategories of $\perv$ and $\ffrm$, respectively.

It follows straightforwardly from Theorem \ref{cfrith} (by identifying
each lattice $S$ with the Frith frame $(\idl(S),S)$) that the
categories of complete Frith frames and of bounded distributive
lattices are equivalent, and by definition of morphism of Frith
frames, this is indeed a full subcategory of $\ffrm$ (see
also~\cite[Proposition~4.6]{borlido21}). Given the Prime Ideal
Theorem, we also have that every $T_0$ complete Pervin space defines a
spectral space.

\begin{lemma}\label{cspectral} If the Prime Ideal Theorem holds, then
  the functor $\uperv:\perv \to \Top$ restricts and co-restricts to a
  functor $\cperv \to \spec$.
\end{lemma}
\begin{proof}
  Let $(X, \cS)$ be a $T_0$ complete Pervin space.  By
  Theorem~\ref{t:1}, $(X, \cS)$ is a fixpoint of the adjunction
  ${\bf \Omega} \dashv \ptf$ and thus, it is sober, that is to say that the
  topological space $(X, \Omega_\cS(X))$ is sober (recall
  Theorem~\ref{dualadj}). By Corollary~\ref{c:4}, the lattice $\cS$ is
  the set of compact open elements of $(X,
  \,\Omega_\cS(X))$. Therefore, $(X, \Omega_\cS(X))$ is
  spectral. Finally, note that this also implies that if $f:
  (X,\cS)\to (Y, \cT)$ is a morphism of $T_0$ complete Pervin spaces,
  then is induces a morphism $f: (X, \Omega_\cS(X)) \to (Y,
  \Omega_\cT(Y))$ of spectral spaces.
\end{proof}
It remains then to show that the categories of $T_0$ complete Pervin
spaces and of spectral spaces are, in fact, isomorphic. Consider the
functor $\ko:\bd{Spec}\ra \cperv$ which sends a spectral space $(X,
\tau)$ to the Pervin space~$X$ equipped with the lattice of compact
open subsets of~$X$. It is easily seen that $\ko$ is well-defined. We
may further prove the following:

\begin{proposition}\label{p:6}
  If the Prime Ideal Theorem holds, then the functors $\uperv$ and
  $\ko$ establish an isomorphism of categories between $\spec$ and
  $\cperv$.
\end{proposition}
\begin{proof}
  Let $(X,\cS)$ be a $T_0$ complete Pervin space. By
  Corollary~\ref{c:4}, $\ko \circ \uperv$ is the identity functor on
  $\cperv$. For a spectral space $(X, \tau)$, we have that $\uperv
  \circ \ko(X, \tau)$ is the set $X$ equipped with the topology
  generated by the lattice of compact open subsets of $(X, \tau)$. But
  by definition of spectral space, this is $(X, \tau)$ itself.
\end{proof}

\begin{corollary}
  If the Prime Ideal Theorem holds, then Stone duality for bounded
  distributive lattices may be seen as a restriction along full
  subcategory embeddings of the dual adjunction ${\bf \Omega}: \perv \lra
  \ffrm:\ptf$.
\end{corollary}
\subsection{Priestley duality}\label{sec:priest}

Another duality for bounded distributive lattices, due to Priestley,
uses the so-called \emph{Priestley spaces} in place of spectral
spaces. A \emph{Priestley space} is a compact topological space
equipped with a partial order relation on its points satisfying the
\emph{Priestley separation axiom}, which states that for points
$x\nleq y$ there is a clopen upper set (``upset" hereon) containing
$x$ and omitting $y$. A morphism of Priestley spaces is a continuous
map which is monotone with respect to the order. We denote by $\pri$
the category of Priestley spaces and corresponding morphisms. It has been shown in \cite{W1975} that the category of Priestley spaces and that of spectral spaces are isomorphic. We have
already seen that $T_0$ complete Pervin spaces form a category
equivalent to spectral spaces, and it is well-known that the latter
form a category equivalent to $\pri$. It is the goal of this section
to explicitly exhibit the correspondence between $T_0$ complete Pervin
spaces and Priestley spaces, thereby providing yet another way of
understanding Priestley duality.

As noticed in~\cite{pin17}, every Pervin space $(X, \cS)$ comes
naturally equipped with a preorder given by
\[
  x\leq_{\cS}y\mb{ if and only if $x\in S$ implies $y\in S$ for all
    $S\in \cS$,}
\]
which is a partial order exactly when $(X, \cS)$ is $T_0$.  In the
case where $(X, \cS)$ is $T_0$ and complete, this is the underlying
partial order of the corresponding Priestley space, and its topology
is the patch topology of the bitopological space $\psk(X, \cS)$.
\begin{lemma}
  If the Prime Ideal Theorem holds, then there is a well-defined
  functor $\pp: \cperv \to \pri$ defined by $\pp(X, \cS) = (X,\,
  \Omega_{\overline{\cS}}(X),\, \leq_\cS)$ on objects, and mapping
  each morphism to the morphism defined by its underlying set
  function.
\end{lemma}
\begin{proof}
  If $(X, \cS)$ is $T_0$ complete, then so is its symmetrization $(X,
  \overline{\cS})$ and, by Corollary~\ref{c:4}, $X \in \overline{\cS}$
  is a compact element of $\Omega_{\overline{\cS}}(X)$. Thus, $(X, \,
  \Omega_{\overline{\cS}}(X))$ is compact. Since elements of~$\cS$ are
  clopen upsets of $(X, \, \Omega_{\overline{\cS}}(X),\, \leq_\cS)$,
  the relation~$\leq_\cS$ satisfies the Priestley separation
  axiom. Therefore, $\pp(X, \cS)$ is a Priestley space. It is not hard
  to see that $\pp$ is well-defined on morphisms, too.
\end{proof}

In the other direction, we have the following:

\begin{lemma}
  There is a well-defined functor $\CUP: \pri \to \cperv$ that assigns
  to each Priestley space $(X, \, \tau, \, \leq)$ the Pervin space $X$
  equipped with the lattice of clopen upsets of~$X$, and keeps
  morphisms unchanged.
\end{lemma}
\begin{proof}
  It is easy to verify that $\CUP$ is well-defined on morphisms. Let
  us argue that $\CUP$ is well-defined on objects. Fix a Priestley
  space $(X, \tau, \leq)$. By the Priestley separation axiom, we have
  that $(X, \cS) := \CUP(X,\tau, \leq)$ is $T_0$. To show that $(X,
  \cS)$ is complete, we show that every Cauchy filter
  converges. Indeed, if $F \subseteq \cP(X)$ is a Cauchy filter then,
  since it is proper, it has the finite intersection property.  Thus,
  $F \cap \overline{\cS}$ is a family of closed subsets of~$X$ with
  the finite intersection property. Since Priestley spaces are
  compact, it follows that $\bigcap (F \cap \overline{\cS})$ is
  nonempty and, by Lemma~\ref{l:7}, $F$ converges.
\end{proof}

We leave it for the reader to verify that the functors $\pp$ and
$\CUP$ are mutually inverse. The reader may also check that, by
composing the functors
\[\pri \xrightarrow{\;\,\CUP\,\;} \cperv \xrightarrow{\,\uperv\,} \spec \qquad
  \text{ and }\qquad \spec \xrightarrow{\quad\pp\quad} \cperv
  \xrightarrow{\;\;\;\ko\;\;\;} \pri\]
one obtains the well-known isomorphism between the categories of
spectral and of Priestley spaces.
\begin{proposition}\label{p:7}
  If the Prime Ideal Theorem holds, then the functors $\pp$ and $\CUP$
  establish an isomorphism of categories between $\pri$ and $\cperv$.
\end{proposition}
\begin{corollary}
  If the Prime Ideal Theorem holds, then Priestley duality may be seen
  as a restriction along full subcategory embeddings of the dual
  adjunction ${\bf \Omega}: \perv \lra \ffrm:\ptf$.
\end{corollary}

\subsection{Bitopological duality}\label{subsec:bitopdua}

It has long been known~\cite{Banaschewski89} that the dual adjunction
between bitopological spaces and biframes restricts and co-restricts
to a duality between $T_0$, compact and zero-dimensional bitopological
spaces (denoted $\kzbitop$) and compact and zero-dimensional biframes
(denoted $\kzbifrm$). A few years later, Priestley duality was
considered from a bitopological point of view~\cite{Picado94}, with
Priestley spaces being identified with $T_0$, compact and
zero-dimensional bitopological spaces, and lattices being identified
with compact and zero-dimensional biframes. The point-set half of this
correspondence was then rediscovered in~\cite{bezhanishvili10}, where
$T_0$ compact zero-dimensional bitopological spaces were named
\emph{pairwise Stone spaces}.\footnote{Note that pairwise Stone spaces
  are the same as $T_0$, compact and zero-dimensional bitopological
  spaces only under the assumption of the Prime Ideal Theorem.} It is
then clear that $\cperv$ and $\cffrm$ are equivalent to
$\kzbitop$ and $\kzbifrm$, respectively.  In this section, we will
make these equivalences explicit, using the adjunctions derived in
Section~\ref{sec:se}.

Let us start with the equivalence between $T_0$ complete Pervin spaces
and $T_0$, compact, and zero-dimensional bitopological spaces. Recall
that we have an idempotent adjunction $\clf_+\dashv \psk$ whose
fixpoints are, respectively, the zero-dimensional bitopological spaces
and the strongly exact Pervin spaces. To conclude that this adjunction
further restricts to an equivalence between $\kzbitop$ and $\cperv$,
it suffices to show that $T_0$ complete Pervin spaces are strongly
exact (cf. Corollary~\ref{c:7}) and that the equality $\psk[\cperv] =
\kzbitop$ holds (cf. Lemmas~\ref{bipit} and~\ref{kc}). For that, we
will need to assume the Prime Ideal Theorem.

\begin{lemma}\label{bipit}
  The following are equivalent.
  \begin{enumerate}
  \item The Prime Ideal Theorem holds.
  \item\label{item:1} For a $T_0$ complete Pervin space $(X,\cS)$, the
    bispace $\psk(X,\cS)$ is compact.
  \item\label{item:14} For a $T_0$ symmetric complete Pervin space
    $(X,\cB)$, the bispace $\psk(X,\cB)$ is compact.
  \end{enumerate}
\end{lemma}
\begin{proof}
  We first argue that~\ref{item:1} and~\ref{item:14} are
  equivalent. Clearly, \ref{item:1} implies~\ref{item:14}. For the
  converse, we only need to remind the reader that a Pervin space $(X,
  \cS)$ is complete if and only if so is its symmetrization $(X,
  \overline{\cS})$ and observe that, by definition of compact bispace,
  $\psk(X, \cS)$ is compact if and only if so is $\psk(X,
  \overline{\cS})$.

  Now, taking the equivalence between statements~\ref{pit2a}
  and~\ref{item:3} of Proposition~\ref{pit2} into account, it suffices
  to show that, for every $T_0$ symmetric Pervin space $(X, \cB)$, the
  bispace $\psk(X, \cB)$ is compact if and only if the Frith frame
  ${\bf \Omega}(X, \cB)$ is coherent (recall that, by
  Theorem~\ref{cfrith}, a Frith frame is coherent if and only if it is
  complete).
  If $\psk(X, \cB)$ is compact, that is, if the topological space $(X,
  \Omega_{\cB}(X))$ is compact, then every element of $\cB$ is compact
  in the frame $\Omega_{\cB}(X)$. Indeed, if $B = \bigcup \cU$ for
  some family $\cU \subseteq \Omega_{\cB}(X)$, then $\cU \cup \{B^c\}$
  is an open cover of $(X, \Omega_{\cB}(X))$ and, if the latter is
  compact, then $\cU \cup \{B^c\}$ has a finite subcover $\cU'$. This
  yields $B = \bigcup \cU' \setminus \{B^c\}$, with $\cU' \setminus
  \{B^c\}$ finite, thereby showing compactness of~$B$. Thus, the Frith
  frame ${\bf \Omega}(X, \cB) = (\Omega_\cB(X), \cB)$ is coherent.
  Conversely, suppose that ${\bf \Omega}(X, \cB)$ is coherent. In
  particular, its underlying frame $\Omega_\cB(X)$ is compact. But
  this is precisely the frame of opens of the patch topology of
  $\psk(X, \cB)$. Thus, the latter is compact as well.
\end{proof}

As a consequence, we immediately obtain that, under the Prime Ideal
Theorem, every $T_0$ complete Pervin space is strongly exact. We do
not know whether this is a necessary hypothesis.

\begin{corollary}\label{c:7}
  If the Prime Ideal Theorem holds, then $T_0$ complete Pervin spaces
  are strongly exact.
\end{corollary}
\begin{proof}
  Let $(X, \cS)$ be a Pervin space.  By Lemma~\ref{bipit}, we only
  need to show that if $\psk(X, \cS)$ is compact then $(X, \cS)$ is
  strongly exact. Let $\bigcap_{i \in I}S_i$ be an open intersection
  of~$\cS$, say $\bigcap_{i \in I}S_i = \bigcup_{j \in J}S_j'$, for
  some $\{S_j\}_{j \in J} \subseteq \cS$. Then, we have
  \[X = (\bigcap_{i \in I}S_i)^c \cup (\bigcap_{i \in I}S_i) =
    \bigcup_{i \in I}S_i^c \cup \bigcup_{j \in J}S_j'\]
  and, by compactness of $\psk(X, \cS)$, it follows that there exists
  a finite subset $J' \subseteq J$ such that $X = \bigcup_{i \in
    I}S_i^c \cup \bigcup_{j \in J'}S_j'$. But then, we have
  $\bigcap_{i \in I} S_i = \bigcup_{j \in J'}S_j'$, and thus, the
  intersection $\bigcap_{i \in I} S_i$ belongs to $\cS$, as required.
\end{proof}

It only remains to show the inclusion $\psk[\cperv] \supseteq
\kzbitop$.

\begin{lemma}\label{kc}
  If $\cX = (X, \tau_+, \tau_-)$ is a compact bitopological space, then the Pervin space $\clf_+(\cX)$ is complete. In
  particular, every $T_0$, compact, and zero-dimensional bitopological
  space is of the form $\psk(X, \cS)$, for some $T_0$ complete Pervin
  space $(X, \cS)$.
\end{lemma}
\begin{proof}
  We let $\tau$ denote the patch topology of $\cX$, and we suppose
  that $\cX$ is compact, that is, that the space $(X, \tau)$ is
  compact. We let $(X, \cS) = \clf_+(\cX)$ and $F \subseteq \cP(X)$ be
  a Cauchy filter. Since $F$ is proper, it has the finite intersection
  property, and therefore, so does $F \cap \overline{\cS}$ (which is a
  subset of $\cl(X, \tau)$). By compactness of $(X, \tau)$ it follows
  that the intersection $\bigcap (F\cap \overline{\cS})$ is
  nonempty. Thus, by Lemma~\ref{l:7}, $F$ converges and $(X, \cS)$ is
  complete.

  Finally, let $\cX = (X, \tau_+, \tau_-)$ be a bitopological
  space. By Proposition~\ref{biadjp}, if $\cX$ is zero-dimensional,
  then it is isomorphic to $\psk \circ \clf_+(\cX)$. If furthermore
  $\cX$ is compact then, by the first part of the claim, $\clf_+(\cX)$
  is complete. It is also easy to see that $\clf_+(\cX)$ is $T_0$
  provided~$\cX$ is $T_0$ and zero-dimensional. Thus, if $\cX$ is
  $T_0$, compact, and zero-dimensional, then $(X, \cS) = \clf_+(\cX)$
  is a $T_0$ complete Pervin space satisfying $\cX \cong \psk(X,
  \cS)$.
\end{proof}

As already explained, we may thus derive the following:
\begin{corollary}\label{c:6}
  If the Prime Ideal Theorem holds, the adjunction $\clf_+\dashv \psk$
  restricts to an equivalence $\kzbitop \cong \cperv$.
\end{corollary}

Let us now consider the pointfree setting. We have seen in
Section~\ref{sec:se} the existence of an idempotent adjunction
$\fsk:\ffrm \lra \bifrm:\bbf_+$. Although we were not able to describe
the underlying categorical equivalence, we have shown that the
fixpoints of $\ffrm$ are the strongly exact Frith frames and that
every compact zero-dimensional biframe is a fixpoint of
$\bifrm$. Thus, we will follow the same strategy as for Pervin spaces
and show that complete Frith frames are strongly exact
(cf. Lemma~\ref{l:2}) and the equality $\fsk[\cffrm] = \kzbifrm$ holds
(cf. Lemmas~\ref{fskcompact} and~\ref{bbcomplete}).

Before proceeding, we need to state a technical result. Recall that a
filter $F \subseteq L$ on a frame~$L$ is \emph{Scott-open} if whenever
$D \subseteq F$ is a directed subset whose join belongs to $F$, the
intersection $D \cap F$ is nonempty.

\begin{proposition}[{\cite{Johnstone85, moshier20}}]\label{p:5}
  Scott-open filters are closed under strongly exact meets.
\end{proposition}
\begin{proof}
  It is shown in~\cite[Lemma~3.4]{Johnstone85} that every Scott-open
  filter $F \subseteq L$ is of the form $\{x \in L \mid \Delta_x
  \subseteq \bigvee_{a \in P}\Delta_a\}$, for some subset $P \subseteq
  L$. In turn, filters of this form are shown
  in~\cite[Theorem~4.5]{moshier20} to be closed under strongly exact
  meets.
\end{proof}

\begin{remark}
  In the proof of Proposition~\ref{p:5} we
  invoked~\cite[Lemma~3.4]{Johnstone85}, whose proof uses ordinal
  induction (although no choice principles are required). A
  constructive alternative proof is provided
  in~\cite[Theorem~1.9]{vickers97}.
\end{remark}

We may now show that complete Frith frames are strongly exact.

\begin{lemma}\label{l:2}
Complete Frith frames are strongly exact.
\end{lemma}
\begin{proof}
  Let $(L,S)$ be a complete Frith frame and $P \subseteq S$ be such
  that $a = \bigwedge P$ is a strongly exact meet. Since, by
  Theorem~\ref{cfrith}, $P$ consists of compact elements and finite
  meets of compact elements are compact too, the filter~$F \subseteq
  L$ generated by~$P$ is Scott-open. By Proposition~\ref{p:5}, $a$
  belongs to~$F$. Thus, there exist $s_1, \dots, s_n \in P$ such that
  $a \geq s_1 \wedge \dots \wedge s_n$, and therefore also $a = s_1
  \wedge \dots \wedge s_n$. This means that $a \in S$, as required.
\end{proof}
\begin{lemma}\label{fskcompact}
  If $(L, S)$ is a complete Frith frame, then the biframe $\fsk(L,S)$
  is compact (and zero-dimensional).
\end{lemma}
\begin{proof}
  By definition, $(L, S)$ is complete if and only if so is $\fsym(L,
  S) = (\cC_SL, \overline{S})$.  By Theorem~\ref{cfrith}, $(\cC_SL,
  \overline{S})$ is complete if and only if it is coherent. In
  particular, if $(L, S)$ is complete, then $\cC_SL$ is compact. But,
  by definition, this means that $\fsk(L, S) = (\cC_SL, \nabla L,
  \Delta S)$ is compact. Finally, $\fsk(L, S)$ is clearly
  zero-dimensional, for every Frith frame $(L, S)$.
\end{proof}

\begin{lemma}\label{bbcomplete}
  If $\cL$ is a compact biframe, then the Frith frame $\bbf_+(\cL)$ is
  complete. In particular, every compact and zero-dimensional biframe
  is of the form $\fsk(L,S)$, for some complete Frith frame~$(L, S)$.
\end{lemma}
\begin{proof}
  Let $\cL = (L, L_+, L_-)$ be a compact biframe and $(M, T) =
  \bbf_+(\cL)$ (that is, $T$ is the lattice of bicomplemented elements
  of~$L_+$ and $M$ is the subframe of $L$ it generates). By
  Theorem~\ref{cfrith}, $\bbf_+(\cL)$ is complete provided $T$ consists
  of compact elements of $M$. Since $M$ is, by definition, a subframe
  of $L_+$, it suffices to show that bicomplemented elements of $L_+$
  are compact in~$L_+$. Let then $t \in L_+$ be bicomplemented and
  suppose that $t \leq \bigvee P$, for some subset $P \subseteq
  L_+$. Then, we have $1 = t \vee t^* \leq \bigvee P \vee t^*$ and,
  since $L$ is compact, there exists a finite subset $P' \subseteq P$
  such that $1 \leq \bigvee P' \vee t^*$. But this implies that $t
  \leq \bigvee P'$, as required.

  Finally, if $\cL$ is compact and zero-dimensional then, by
  Proposition~\ref{p:8}, $\cL \cong \fsk\circ \bbf_+(\cL)$ and, in
  particular, it is of the form $\fsk(L, S)$ for some complete Frith
  frame $(L, S)$.
\end{proof}

We may thus conclude the following:
\begin{corollary}\label{c:5}
  The adjunction $\fsk\dashv \bbf_+$ restricts to an equivalence
  $\cffrm \cong \kzbifrm$.
\end{corollary}

\begin{corollary}
  If the Prime Ideal Theorem holds, then the duality between compact
  zero-dimensional bitopological spaces and compact zero-dimensional
  biframes may be seen as restriction and co-restriction of the dual
  adjunction ${\bf \Omega}: \perv \lra \ffrm:\ptf$ along full
  subcategory embeddings.
\end{corollary}

We remark that, in~\cite{Picado94}, the duality between compact
zero-dimensional bispaces and compact zero-dimensional biframes is
seen as a restriction and co-restriction of the classical dual
adjunction ${\bf \Omega}_b: \bitop \lra \bifrm: \ptf_b$. However,
although bitopological spaces and biframes are related to Pervin spaces
and Frith frames (namely, via the functors $\psk$ and $\fsk$,
respectively), as seen in Example~\ref{sec:2}, the two extensions of
the adjunction ${\bf \Omega}: \Top \lra \Frm : \ptf$ are not
comparable. The situation is different when we restrict to $T_0$
complete and to $T_0$ compact and zero-dimensional structures, as in
that case, both ${\bf \Omega}_b: \bitop \lra \bifrm: \ptf_b$ and ${\bf
  \Omega}: {\bf C\perv} \lra {\bf C\ffrm}: \ptf$ coincide, up to
equivalence. Therefore, our duality ends up being equivalent to the duality of~\cite{Picado94}.

\subsection{Summary}\label{sec:overview}

We summarize in the diagram below the categorical equivalences shown
in Sections~\ref{sec:stone}--\ref{subsec:bitopdua}, which exhibit
several Stone-type dualities as restrictions of the dual adjunction
between Pervin spaces and Frith frames. The results marked with~$*$
require the use of the Prime Ideal Theorem.

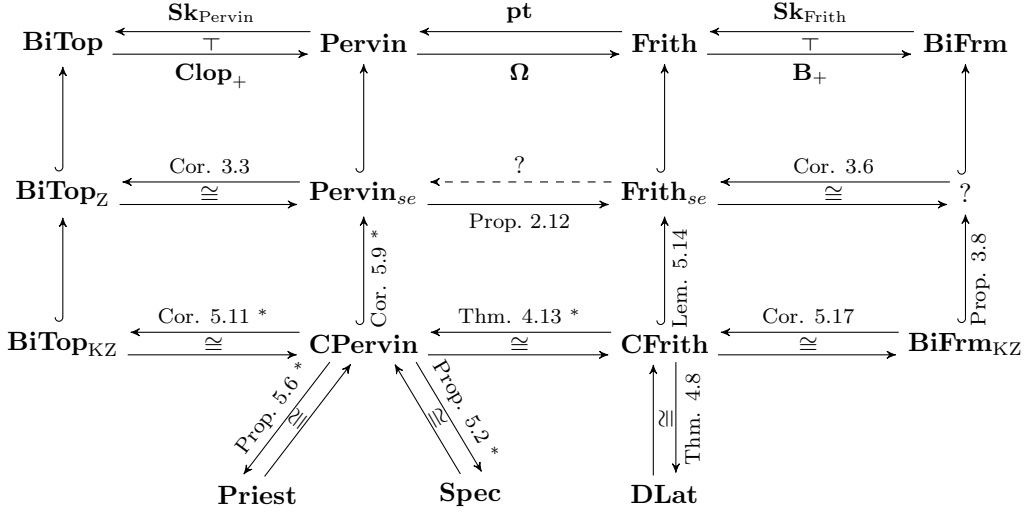
\begin{center}
  \begin{tikzpicture}[node distance = 20mm]
    \node (bitop) {$\bitop$};
    \node[right of = bitop, xshift = 20mm] (perv) {$\perv$};
    \node[right of = perv, xshift = 20mm] (frith) {$\ffrm$};
    \node[right of = frith, xshift = 20mm] (bifrm) {$\bifrm$};
    \node[below of = bitop] (bitop0) {$\bitop_{\rm Z}$};
    \node[below of = bitop0] (bitopKZ) {$\bitop_{\rm KZ}$};
    \node[below of = bifrm] (fix) {?};
    \node[below of = fix] (bifrmKZ) {$\bifrm_{\rm KZ}$};
    \node[below of = perv] (pervse) {$\perv_{se}$};
    \node[below of = pervse] (cperv) {$\bf C\perv$};
    \node[below of = frith] (frithse) {$\ffrm_{se}$};
    \node[below of = frithse] (cfrith) {$\bf C\ffrm$};
    \node[below of = cfrith] (dlat) {$\dlat$};
    \node[below right of = cperv, yshift = -6mm] (spec) {$\spec$};
    \node[below left of = cperv, yshift = -6mm] (priest) {$\pri$};
    \draw ([xshift=-1.5mm]dlat.90) to node[ArrowNode, right, xshift
    =-.5mm]{\rotatebox{90}{$\cong$}} ([xshift=-1.5mm]cfrith.270); \draw
    ([xshift=1.5mm]cfrith.270) to node[ArrowNode,right]
    {\rotatebox{90}{\scriptsize Thm. \ref{cfrith}}} ([xshift=1.5mm]dlat.90);
    \draw ([xshift=1.5mm]spec.135) to node[ArrowNode, right, xshift
    =-1.5mm, yshift =0mm]{\rotatebox{-60}{$\cong$}}
    ([xshift=1.5mm]cperv.315); \draw ([xshift=4.5mm]cperv.315) to
    node[ArrowNode,right, xshift =-4mm, yshift =1mm]
    {\rotatebox{-60}{\scriptsize Prop. \ref{p:6} $^*$}}
    ([xshift=4.5mm]spec.135);
    \draw ([xshift=-1.5mm]priest.45) to node[ArrowNode, left, xshift
    =1.5mm]{\rotatebox{53}{$\cong$}} ([xshift=-1.5mm]cperv.270); \draw
    ([xshift=-4.5mm]cperv.270) to node[ArrowNode,left, xshift = 5mm,
    yshift = 1.5mm]
    {\rotatebox{53}{\scriptsize Prop.~\ref{p:7} $^*$}}
    ([xshift=-4.5mm]priest.45);
    \draw[right hook->] (bitopKZ) to (bitop0); \draw[right hook->]
    (bitop0) to (bitop); \draw[right hook->] (cperv) to node[above,
    yshift = -9mm, xshift = 2mm] {\rotatebox{90}{\scriptsize
        Cor. \ref{c:7} $^*$}} (pervse); \draw[right hook->] (pervse)
    to (perv); \draw[right hook->] (cfrith) to node[above, yshift =
    -9mm, xshift = 2mm] {\rotatebox{90}{\scriptsize Lem. \ref{l:2}}}
    (frithse); \draw[right hook->] (frithse) to (frith); \draw[right
    hook->] (bifrmKZ) to node[above, yshift = -9mm, xshift = 2mm]
    {\rotatebox{90}{\scriptsize Prop. \ref{p:8}}} (fix); \draw[right
    hook->] (fix) to (bifrm);
    \draw ([yshift=-1.5mm]bitop.0) to node[ArrowNode,below] {$\clf_+$}
    node[ArrowNode, above, yshift =-1mm] {\rotatebox{-90}{$\vdash$}}
    ([yshift=-1.5mm]perv.180);
    \draw ([yshift=1.5mm]perv.180) to
    node[ArrowNode,above] {$\psk$} ([yshift=1.5mm]bitop.0);
    \draw ([yshift=-1.5mm]perv.0) to node[ArrowNode,below] {${\bf
        \Omega}$} ([yshift=-1.5mm]frith.180); \draw
    ([yshift=1.5mm]frith.180) to node[ArrowNode,above] {$\ptf$}
    ([yshift=1.5mm]perv.0);
    \draw ([yshift=-1.5mm]frith.0) to node[ArrowNode,below] {$\bbf_+$}
    node[ArrowNode, above, yshift =-1mm] {\rotatebox{-90}{$\vdash$}}
    ([yshift=-1.5mm]bifrm.180);
    \draw ([yshift=1.5mm]bifrm.180) to
    node[ArrowNode,above] {$\fsk$} ([yshift=1.5mm]frith.0);
    \draw ([yshift=-1.5mm]pervse.0) to node[ArrowNode,below]
    {\scriptsize Prop.~\ref{p:9}} ([yshift=-1.5mm]frithse.180);
    \draw[dashed] ([yshift=1.5mm]frithse.180) to node[ArrowNode,above]
    {$?$} ([yshift=1.5mm]pervse.0);
    \draw ([yshift=-1.5mm]bitop0.0) to node[ArrowNode, above, yshift
    =-1mm] {$\cong$} ([yshift=-1.5mm]pervse.180); \draw
    ([yshift=1.5mm]pervse.180) to node[ArrowNode,above] {\scriptsize
      Cor. \ref{pointsetworks}} ([yshift=1.5mm]bitop0.0);
    \draw ([yshift=-1.5mm]frithse.0) to node[ArrowNode, above, yshift
    =-1mm] {$\cong$} ([yshift=-1.5mm]fix.180); \draw
    ([yshift=1.5mm]fix.180) to node[ArrowNode,above] {\scriptsize
      Cor. \ref{c:2}} ([yshift=1.5mm]frithse.0);
    \draw ([yshift=-1.5mm]bitopKZ.0) to node[ArrowNode, above, yshift
    =-1mm]{$\cong$} ([yshift=-1.5mm]cperv.180); \draw
    ([yshift=1.5mm]cperv.180) to node[ArrowNode,above] {\scriptsize
      Cor. \ref{c:6} $^*$} ([yshift=1.5mm]bitopKZ.0);
    \draw ([yshift=-1.5mm]cperv.0) to node[ArrowNode, above, yshift
    =-1mm]{$\cong$} ([yshift=-1.5mm]cfrith.180); \draw
    ([yshift=1.5mm]cfrith.180) to node[ArrowNode,above] {\scriptsize
      Thm. \ref{t:1} $^*$} ([yshift=1.5mm]cperv.0);
    \draw ([yshift=-1.5mm]cfrith.0) to node[ArrowNode, above, yshift
    =-1mm] {$\cong$} ([yshift=-1.5mm]bifrmKZ.180); \draw
    ([yshift=1.5mm]bifrmKZ.180) to node[ArrowNode,above] {\scriptsize
      Cor. \ref{c:5}} ([yshift=1.5mm]cfrith.0);
  \end{tikzpicture}
\end{center}

We highlight that the explicit equivalences between $\bitop_{\rm KZ}$ and $\pri$, and between $\bifrm_{\rm KZ}$ and $\dlat$ have first been studied in~\cite{Picado94}.

\section{The $T_D$ axiom for Pervin spaces and locale-based Frith
  frames}\label{sec:last}

In \cite[Section~4.5]{borlido21} a Pervin equivalent of the $T_D$ axiom for topological
spaces is introduced. A Pervin space is called \emph{$T_D$} if for every
$x\in X$ there is $S\in \cS$ such that $x\in S$ and $S{\sm}\{x\}\in
\cS$. In there, the following slight variants of this definition were
shown to be equivalent:
\begin{lemma}\label{l:10}
  Let $(X, \cS)$ be a Pervin space. Then, the following are equivalent:
  \begin{enumerate}
  \item $(X, \cS)$ is $T_D$,
  \item\label{item:8} for every $x \in X$, there are distinct $S_1,
    S_2 \in \cS$ such that $S_1 \setminus \{x\} = S_2 \setminus
    \{x\}$,
  \item\label{item:7} for every $x \in X$, there are $S_1, S_2 \in
    \cS$ such that $S_1 \setminus S_2 = \{x\}$.
  \end{enumerate}
\end{lemma}
On the other hand, in~\cite{banaschewski10} the $T_D$ axiom is
compared to sobriety and it is highlighted that the two axioms are, in
some sense, duals of each other. In particular, the following is
proven.

\begin{proposition}[{\cite[Proposition 4.3]{banaschewski10}}]\label{p:3}
  Let $X$ be a topological space. In the category of $T_0$ topological
  spaces we have the following:
\begin{enumerate}
\item The space $X$ is sober if and only if whenever we have an
  extremal monomorphism $f:X\hookrightarrow Y$ such that ${\bf
    \Om}(f)$ is an isomorphism $f$ must be an isomorphism.
\item The space $X$ is $T_D$ if and only if whenever we have an
  extremal monomorphism $f:Y\hookrightarrow X$ such that ${\bf
    \Om}(f)$ is an isomorphism $f$ must be an isomorphism.
\end{enumerate}
\end{proposition}
We will now show a Pervin version of this proposition, where
\emph{sober} is replaced by \emph{complete} and the \emph{$T_D$-axiom}
is replaced by its equivalent for Pervin spaces, being the next lemma the
key ingredient. Recall the forgetful function $\lperv: \perv \to
\dlat^{op}$ introduced in Section~\ref{sec:forget}.

\begin{lemma}\label{l:11}
  Let $f: (X, \cS) \to (Y, \cT)$ be a morphism of Pervin spaces. Then,
  \begin{enumerate}
  \item\label{item:5} $\lperv(f):\cT \to \cS$ is injective if and only
    if $\psym(f): (X, \overline{\cS}) \to (Y, \overline{\cT})$ is
    dense;
  \item\label{item:6} $\lperv(f):\cT \to \cS$ is surjective if and
    only if ${\bf \Omega}(f): (\Omega_\cT(Y), \cT)\to (\Omega_\cS(X),
    \cS) $ is an extremal epimorphism of Frith frames. Moreover, if
    $(X, \cS)$ is $T_0$, these are further equivalent to $f$ being an
    extremal monomorphism of Pervin spaces.
  \end{enumerate}
\end{lemma}
\begin{proof}
  Let us start by proving~\ref{item:5}. Since $\lperv(f)$ is a lattice
  homomorphism, being injective is equivalent to having
  \[\forall T_1, T_2 \in \cT, \ f^{-1}(T_1) \subseteq f^{-1}(T_2)
    \implies T_1 \subseteq T_2,\]
  which is easily seen to be equivalent to having
  \[\forall T_1, T_2 \in \cT, \ f^{-1}(T_1 \cap T_2^c) = \emptyset
    \implies T_1 \cap T_2^c = \emptyset.\]
  Since $\overline{\cT}$ consists of the finite joins of elements of
  the form $T_1 \cap T_2^c$, with $T_1, T_2 \in \cT$, this means that
  $\psym(f)$ is dense, as required.

  Finally, the first assertion of~\ref{item:6} is a trivial
  consequence of the characterization of extremal epimorphisms of
  Frith frames (recall Proposition~\ref{p:4}), while the second one is
  the content of~\cite[Corollary~4.17]{borlido21}.
\end{proof}

\begin{proposition}
  \label{td}
  Let $(X,\cS)$ and $(Y, \cT)$ be $T_0$ Pervin spaces. We have the
  following.
  \begin{enumerate}
  \item \label{tda}The space $(X, \cS)$ is complete if and only if
    whenever there is a map $f:(X,\cS)\ra (Y,\cT)$ such that
    $\lperv(f)$ is an isomorphism $f$ must be an isomorphism.
  \item \label{tdb}The space $(X, \cS)$ is $T_D$ if and only if
    whenever there is a map $f:(Y,\cT)\ra (X,\cS)$ such that
    $\lperv(f)$ is an isomorphism $f$ must be an isomorphism.
  \end{enumerate}
\end{proposition}
\begin{proof}
  Since $(X, \cS)$ is $T_0$, by Lemma~\ref{l:11}, $\lperv(f)$ is an
  isomorphism if and only if $f$ is an extremal monomorphism and
  $\psym(f)$ is dense. Thus, part~\ref{tda} follows from
  Theorem~\ref{char}\ref{char3}.

  To prove~\ref{tdb}, let us first suppose that $(X, \cS)$ is
  $T_D$ and let $f: (Y, \cT) \to (X, \cS)$ be such that $(Y,
  \cT)$ is $T_0$ and $\lperv(f)$ is an isomorphism. Since $(Y, \cT)$
  is $T_0$, by Lemma~\ref{l:11}\ref{item:6}, $f$ is an extremal
  monomorphism. Thus, it suffices to show that $f$ is an epimorphism,
  that is, surjective. Fix $x \in X$. Since $(X, \cS)$ is
  $T_D$, by Lemma~\ref{l:10}\ref{item:7}, the singleton
  $\{x\}$ belongs to $\overline{\cS}$. Since, by
  Lemma~\ref{l:11}\ref{item:5}, $\psym(f)$ is dense, it follows that
  $\{x\}$ intersects $f[Y]$ (recall Lemma~\ref{l:9}), that is, $x$
  belongs to the image of $f$ as we intended to show. Conversely, let
  $f: (Y, \cT) \hookrightarrow (X, \cS)$ be the extremal monomorphism
  induced by $Y := X\setminus \{x\}$, that is, $\cT = \{S \setminus
  \{x\} \mid S \in \cS\}$. Clearly, $(Y, \cT)$ is $T_0$ and $f$ is not
  an isomorphism. Thus, by assumption, $\lperv(f)$ cannot be
  injective. That is to say that, there exist distinct $S_1, S_2 \in
  \cS$ such that $S_1\setminus \{x\} = S_2 \setminus\{x\}$ and thus,
  by Lemma~\ref{l:10}\ref{item:8}, $(X, \cS)$ is $T_D$.
\end{proof}

We finish this section by exhibiting a pointfree version of
Proposition~\ref{td}, which will now involve the forgetful functor
$\lfrith: \ffrm \to \dlat$ (cf. Section~\ref{sec:forget}). Let us
start with a version of Lemma~\ref{l:11} for Frith frames.

\begin{lemma}\label{l:12}
  Let $h: (L, S) \to (M, T)$ be a homomorphism of Frith frames. Then,
  \begin{enumerate}
  \item\label{item:9} $\lfrith(h):S \to T$ is injective if and only if
    $\fsym(h): (\cC_SL, \overline{S}) \to (\cC_TM, \overline{T})$ is
    dense;
  \item\label{item:10} $\lfrith(h): S \to T$ is surjective if and only
    if $h$ is an extremal epimorphism.
  \end{enumerate}
\end{lemma}
\begin{proof} By definition of $\fsym$, we have that $\fsym(h)$ is
  dense if and only if for every $s_1, s_2 \in S$, having
  $\nabla_{h(s_1)}\wedge \Delta_{h(s_2)} = 0$ implies that
  $\nabla_{s_1} \wedge \Delta_{s_2} = 0$. But this is easily seen to
  be equivalent to having that $h(s_1)\leq h(s_2)$ implies that $s_1
  \leq s_2$, that is to say that $\lfrith (h)$ is injective. This
  proves~\ref{item:9}.  Part~\ref{item:10} is a straightforward
  consequence of the definition of $\lperv(h)$ and of the
  characterization of extremal epimorphisms of Frith frames
  (cf. Proposition~\ref{p:4}).
\end{proof}

We may then prove the following version of
Proposition~\ref{td}\ref{tda}. In particular, note that it provides
an alternative criterion for a Frith frame to be complete, which does
not depend on its symmetrization.
\begin{proposition}
  \label{tdfrm}
  A Frith frame $(L,S)$ is complete if and only if for every map
  $h:(M,T)\ra (L,S)$ such that $\lfrith(h)$ is an isomorphism, $h$ is
  an isomorphism.
\end{proposition}
\begin{proof}
  We first observe that $(L, S)$ is complete if and only if every
  dense extremal epimorphisms $(K, C) \to (\cC_SL, \overline{S})$,
  with $(K, C)$ symmetric, is an isomorphism. Suppose that $(L, S)$ is
  complete and let $h:(M,T)\ra (L,S)$ be such that $\lfrith(h)$ is an
  isomorphism. By Lemma~\ref{l:12}\ref{item:10}, $h$ is an extremal
  epimorphism. Thus, we only need to show that it is also a
  monomorphism, that is, injective. Consider the homomorphism
  $\fsym(h): (\cC_TM, \overline{T}) \to (\cC_SL, \overline{S})$. Then,
  $\fsym(h)$ is an extremal epimorphism because so is $h$ and, by
  Lemma~\ref{l:12}\ref{item:9}, it is dense. Thus, since $(L, S)$ is
  complete, it is an isomorphism and, in particular, injective. Thus,
  its restriction $h$ is also injective as required. Conversely, let
  $h:(K, C) \to (\cC_SL, \overline{S})$ be a dense extremal
  epimorphism, with $(K, C)$ symmetric. Since $h$ is already a
  morphism in $\sffrm$, and thus, $\fsym(h) = h$, by Lemma~\ref{l:12},
  $\lfrith(h)$ is an isomorphism. Thus, by hypothesis, $h$ is an
  isomorphism, which proves that $(L, S)$ is complete.
\end{proof}

In order to get an analogue of Proposition~\ref{td}\ref{tdb}, we need
to introduce the notion of \emph{locale-based} Frith frame, which will
replace the $T_D$ axiom in our statement.

We say that a Frith frame $(L, S)$ is \emph{locale-based} if the
smallest sublocale of $L$ that contains~$S$ is the whole frame $L$. It
is not hard to see that such sublocale consists of the arbitrary meets
of elements of the form $a \to s$, where $a \in L$ and $s \in S$. As a
consequence, we have the following characterization of locale-based
Frith frames, that we state for later reference.

\begin{lemma}\label{l:14}
  A Frith frame is locale-based if and only if, for every $a \in L$,
  the following equality holds:
  \[a = \bigwedge \{b \to s \mid b \in L, \, s \in S, \text{ and } a
    \leq b \to s\}.\]
\end{lemma}
The following technical result will also be useful:

\begin{lemma}\label{l:13}
  Let $h: (L, S) \to (M, T)$ be a morphism of Frith frames. If
  $\lfrith(h)$ is injective, then $h_* \circ h (s) = s$, for every $s
  \in S$.
\end{lemma}
\begin{proof}
  Since $h_*$ is right adjoint to~$h$, we have $h \circ h_* \circ h =
  h$. Thus, the claim follows from injectivity of $\lfrith(h)$.
\end{proof}

We are now ready to show the already announced analogue of
Proposition~\ref{td}\ref{tdb}.

\begin{proposition}\label{cotdfrm}
  A Frith frame $(L,S)$ is locale-based if and only if for every map
  $h:(L,S)\ra (M,T)$ such that $\lfrith(h)$ is an isomorphism, $h$ is
  an isomorphism.
\end{proposition}
\begin{proof}
  Let $(L, S)$ be a locale-based Frith frame, and $h: (L, S) \to (M,
  T)$ be a homomorphism such that $\lfrith(h)$ is an isomorphism. By
  Lemma~\ref{l:12}\ref{item:10}, we know that $h$ is an extremal
  epimorphism. It remains to show it is injective or, equivalently,
  that $h_* \circ h (a) \leq a$, for every $a \in L$.  It remains to show it is injective or, equivalently,
  that $h_* \circ h (a) \leq a$, for every $a \in L$. By Lemma \ref{l:13} we know that $S\se h_*[M]$. This is also a sublocale inclusion, as $S\se L$ is, $L$ being locale-based. As $h_*[M]\se L$ is a sublocale then it must be $L$, as $L$ is locale-based. 

  For the converse, suppose that $(L,S)$ is such that all maps
  $h:(L,S)\ra (M,T)$ such that $\lfrith(h)$ is an isomorphism are
  isomorphisms. We let $\langle S\rangle_{\bf Loc}$ be the smallest
  sublocale of $L$ that contains $S$. Since $\langle S\rangle_{\bf
    Loc}$ is a subposet of $L$, each join computed in $\langle
  S\rangle_{\bf Loc}$ is greater than or equal to the same join
  computed in $L$. Thus, $(\langle S\rangle_{\bf Loc}, S)$ is also a
  Frith frame. Finally, we let $q: L \twoheadrightarrow \langle
  S\rangle_{\bf Loc}$ be the frame quotient defined by the sublocale
  embedding $\langle S\rangle_{\bf Loc} \hookrightarrow L$. Since the
  restriction of $q$ to $\langle S\rangle_{\bf Loc}$ is the identity,
  $q$ induces a morphism of Frith frames $q: (L, S) \to (\langle
  S\rangle_{\bf Loc}, S)$. Clearly we have that $\lfrith(q)$ is an
  isomorphism, and so, by hypothesis, $h$ is an isomorphism. This
  proves that $(L,S)$ is locale-based, as required.
\end{proof}

\section{Declarations}
\subsection{Author’s Contribution}
Both authors contributed to all stages of the paper.
\subsection{Conflict of Interests}
None
\subsection{Availability of Data and Materials}
Not applicable
\subsection{Funding}
C\'{e}lia Borlido was partially supported by the Centre for Mathematics of the University of Coimbra -
UIDB/00324/2020, funded by the Portuguese Government through FCT/MCTES. Anna Laura Suarez received funding from the European Research Council (ERC) under the European Union’s Horizon 2020 research and innovation
program (grant agreement No.670624).

\bibliographystyle{acm}

\end{document}